\documentclass[preprint,12pt]{elsarticle}
\usepackage{xy}
\usepackage{amssymb}
\xyoption{all }
\usepackage{amsthm}
\usepackage{graphicx}
\usepackage{lineno}

\def\mc#1{\mathcal {#1}}
\def\C{\mc C}
\def\D{\mc D}
\def\S{\mc S}
\def\X{\mc X}
\def\Y{\mc Y}
\def\P{\mc P}
\def\Q{\mc Q}
\def\T{\mc T}

\def\I{\mc I}
\def\hs#1{\hskip#1mm}

\newtheorem{lemma}{Lemma}
\newtheorem{corollary}{Corollary}
\newtheorem{theorem}{Theorem}
\newtheorem{proposition}{Proposition}
\newtheorem{example}{Example}
\newtheorem{definition}{Definition}

\newtheorem{remark}{Remark}
\newtheorem{open problem}{Open Problem}

\begin{document}
\begin{frontmatter}
\title{Fraction, Restriction, and Range Categories\\ From Non-Monic Classes of Morphisms
 \tnoteref{ANW}}
\tnotetext[ANW]{\scriptsize Much of the work presented in this paper was done while the second author visited York University in 2018, with the partial financial assistance from the third author's NSERC Discovery Grant (no. 050126), which is gratefully acknowledged. The first two authors acknowledge also the partial support by Mahani Mathematical Research Center.}
\author[1]{S.N. Hosseini}
\ead{nhoseini@uk.ac.ir}%
\author[1]{A.R. Shir Ali Nasab}
\ead{ashirali@math.uk.ac.ir}
\author[2]{W. Tholen\corref{W}}
\ead{tholen@mathstat.yorku.ca}
\address[1]{Mahani Mathematical Research Center\\ Shahid Bahonar University of Kerman, Kerman, Iran}
\address[2]{Department of Mathematics and Statistics\\ York University, Toronto, Canada}
\cortext[W]{\scriptsize Corresponding author}

 \begin{abstract}
{\footnotesize  For a composition-closed and pullback-stable class $\mathcal S$ of morphisms in a category $\mathcal C$ containing all isomorphisms, we form the category {\sf Span}$(\mathcal C,\mathcal S)$
of $\mathcal S$-spans $(s,f)$ in $\mathcal C$ with first ``leg" $s$ lying in $\mathcal S$, and give an alternative construction of its quotient category $\mathcal C[\mathcal S^{-1}]$ of $\mathcal S$-fractions. Instead of trying to turn $\mathcal S$-morphisms ``directly" into isomorphisms, we turn them separately into retractions and into sections in a universal manner, thus obtaining the quotient categories {\sf Retr}$(\mathcal C,\mathcal S)$ and {\sf Sect}$(\mathcal C,\mathcal S)$. The fraction category $\mathcal C[{\mathcal S}^{-1}]$ is their largest joint quotient category. 

Without confining $\mathcal S$ to be a class of monomorphisms of $\mathcal C$, we show that {\sf Sect}$(\mathcal C,\mathcal S)$ admits a quotient category, {\sf Par}$(\mathcal C,\mathcal S)$, whose name is justified by two facts. On one hand, for $\mathcal S$ a class of monomorphisms in $\mathcal C$, it returns the category of $\mathcal S$-spans in $\mathcal C$, also called $\mathcal S$-partial maps in this case; on the other hand, we prove
that {\sf Par}$(\mathcal C,\mathcal S)$ is a split restriction category (in the sense of Cockett and Lack).  A further quotient construction produces even a range category (in the sense of Cockett, Guo and Hofstra), {\sf RaPar}$(\mathcal C,\mathcal S)$, which is still large enough to admit $\mathcal C[\mathcal S^{-1}]$ as its quotient. 

Both, {\sf Par} and {\sf RaPar}, are the left adjoints of global 2-adjunctions. When restricting these to their ``fixed objects", one obtains precisely the 2-equivalences by which their name givers characterized restriction and range categories. Hence, both {\sf Par}$(\mathcal C,\mathcal S)$ and {\sf RaPar}$(\mathcal C,\mathcal S)$ may be naturally presented as {\sf Par}$(\mathcal D,\mathcal T)$ and {\sf RaPar}$(\mathcal D,\mathcal T)$, respectively, where now $\mathcal T$ is a class of monomorphisms in $\mathcal D$. In summary, while there is no {\em a priori} need for the exclusive consideration of classes of monomorphisms, one may resort to them naturally.}
\end{abstract}

{\tiny 
 }
\end{frontmatter}


 \section{Introduction}
The formation of the {\em category} $\;\C[\S^{-1}]$ {\em of fractions} with respect to a sufficiently well-behaved class $\S$ of morphisms in $\C$, as first given in \cite{gabzis}, is a fundamental device in homotopy theory.
The construction is characterized by its {\em localizing functor} $\C\to\C[\S^{-1}]$ which is universal with respect to the property of turning morphisms in $\S$ into isomorphisms. An existence proof for $\C[\S^{-1}]$ is only sketched by Gabriel and Zisman \cite{gabzis} (see their Lemma 1.2 on p. 7 whose {\em ``proof is left to the reader"}); for more elaborate proofs, see \cite{schubert} and \cite{bor}. A particular and delicate point is the question of the size of the ``homs" of $\C[\S^{-1}]$, as these may be large even when those of $\C$ are all small.

 Assuming $\S$ {\em to contain all isomorphisms and to be closed under composition and stable under pullback in} $\C$ throughout this paper and, thus, departing from the original array of applications for the construction, we take a stepwise approach to the formation of $\C[\S^{-1}]$. Hence, we consider separately the two processes of transforming every morphism in $\S$ into a retraction and into a section, before amalgamating them to obtain the category of fractions. Not surprisingly, when $\S$ happens to be a class of monomorphisms in $\C$, the transformation of $\S$-morphisms into retractions essentially suffices to reach $\C[\S^{-1}]$, simply because the transformation of $\S$-morphisms into sections comes almost for free when $\S$ is a class of monomorphisms: one just considers the $\S$-{\em span category} $\sf{Span}(\C,\S)$ whose morphisms $(s,f):A\to B$ are (isomorphism classes of) spans
	$\xymatrix{A & D\ar[l]_{s}\ar[r]^{f} & B}$
of morphisms in $\C$ with $s\in\S$; composition with $(t,g):B\to C$ proceeds as usual, via pullback:
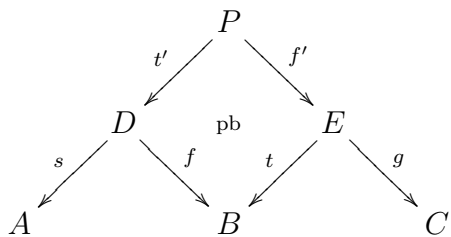
\begin{center}
	$\xymatrix{&& P\ar[ld]_{t'}\ar[rd]^{f'} && \\
	& D\ar[ld]_{s}\ar[rd]^{f}\ar@{}[rr]|{\rm{pb}} & & E\ar[ld]_{t}\ar[rd]^{g} & \\
A && B && C}$
\end{center}
Trivially now, the functor $\C\to \sf{Span}(\C,\S),\,f\mapsto (1,f),$ turns $\S$-morphisms into sections, since monomorphisms have trivial kernel pairs (in the diagram above, for $s=g=1$ and $f=t\in\S$ one can take $t'=f'=1$).

 In the general case, {\em without} confining $\S$ to be a class of monomorphisms, as a first step we will still form the category $\sf{Span}(\C,\S)$ as above. Then, transforming $\S$-morphisms into retractions in a universal manner is fairly easy; however,  trying to transform them into
sections without the assumption that they are monomorphisms is considerably more complicated. 
The latter problem leads us to the formation of the $\S$-{\em sectionable span category} of $\C$, {\sf Sect}$(\C,\S)$, 
while the former problem makes us form the $\S$-{\em retractable span category} of $\C$, $\sf{Retr}(\C,\S)$ (see Sections 3 and 2, respectively). In Section 4 we see how to amalgamate the two constructions to obtain the category $\C[\S^{-1}]$. Although, in order to keep the paper widely accessible, we perform and characterize these constructions strictly at the ordinary-category level, in an epilogue (Section 10) for readers familiar with bicategories, we briefly allude to the 2-categorical structure of $\sf{Span}(\C,\S)$ (see \cite{Benabou, Hermida}) 
and indicate how our constructions of $\sf{Retr}(\C,\S)$ and {\sf Sect}$(\C,\S)$ are naturally motivated by that context.

 Whilst the fact that $\C\to{\sf Sect}(\C,\S)$ is universal with respect to turning $\S$-morphisms into sections makes the category a candidate for serving as an $\S$-partial map category, 
without the mono constraint on $\S$ we have not been able to determine whether it is a {\em restriction category}, {\em i.e.}, whether it has the property identified by Cockett and Lack \cite{colackrest} as characteristic for $\S$-partial map categories in the case that $\S$ is a class of monomorphisms.
That is why, in Section 5, we elaborate on how to obtain, as a quotient category of {\sf Sect}$(\C,\S)$, the $\S$-{\em partial map category} ${\sf Par}(\C,\S)$, which is a restriction category. 
 Under a fairly mild additional hypothesis on $\S$, which holds in particular under the {\em weak left cancellation condition} $(s,\, s\cdot t\in\S\Longrightarrow t\in \S)$, ${\sf Par}(\C,\S)$ is a localization of ${\sf Sect}(\C,S)$ and makes ${\sf Retr}(\C,\S)=\C[\S^{-1}]$ its quotient category.

 Of course, this additional condition holds {\em a fortiori} when $\S$ belongs to a {\em relatively stable} orthogonal factorization system $(\mathcal P,\S)$ of $\C$, so that $\P$ is stable under pullback along $\S$-morphisms. In that case we can form, as a further localization of $\sf{Sect}(\C,\S)$, a {\em range category} in the sense of \cite{cgh}. Range categories not only have a restriction structure on the domains of morphisms, but also a kind of dually behaved structure on their codomains. Hence, in Section 8 we present the construction of the $\S$-{\em partial map range category}, ${\sf RaPar}(\C,S)$, thus completing the quotient constructions given in this paper.

 In summary, for $\S$ satisfying the general hypotheses one has the commutative diagram
$$\xymatrix{\C\ar[r] & {\sf Span}(\C,\S)\ar[r]\ar[d] & {\sf Sect}(\C,\S)\ar[d]\ar[r] & {\sf Par}(\C,\S) \\
			 & {\sf Retr}(\C,\S)\ar[r] & \C[\S^{-1}]}$$
			which flattens to				
$$C\to{\sf Span}(\C,\S)\to{\sf Sect}(\C,\S)\to{\sf Par}(\C,\S)\to{\sf Retr}(\C,\S)=\C[\S^{-1}]$$
when $\S$ satisfies the weak left cancellation property, and it extends further to
$${\sf Span}(\C,\S)\to{\sf Sect}(\C,\S)\to{\sf Par}(\C,\S)\to{\sf RaPar}(\C,\S)\to{\sf Retr}(\C,\S)=\C[\S^{-1}]$$
 when $\S$ belongs to an $\S$-stable factorization system $(\mathcal P,\S)$ of $\C$.
When $\S$ is a class of monomorphisms, the chain simplifies to
$$C\to{\sf Span}(\C,\S)={\sf Sect}(\C,\S)={\sf Par}(\C,\S)\to{\sf Retr}(\C,\S)=\C[\S^{-1}],$$
and one then has also \,${\sf Par}(\C,\S) ={\sf RaPar}(\C,\S)$,
should $\S$ be part of an $\S$-stable factorization system $({\mathcal P},\S)$.

 Quite a different picture emerges when one puts additional constraints on $\S$ that are typically satisfied by classes of epimorphisms, rather than monomorphisms. In Sections 4 and 5 we show that, when $\C$ has finite products with all projections lying in $\S$, and if there is no strict initial object in $\C$, then ${\sf Par}(\C.\S)$ is equivalent to the terminal category $\sf 1$, and one has
$$C\to{\sf Span}(\C,\S)\to{\sf Retr}(\C,\S)\to{\sf Sect}(\C,\S)={\sf Par}(\C,\S)=\C[\S^{-1}]\simeq{\sf 1}.$$

The ultimate justification for the formation of {\sf Par}$(\C,S)$ and {\sf RaPar}$(\C,\S)$ lies in their universal role, as presented in Sections 7 and 9, respectively. In fact, {\sf Par} is the left adjoint of a 2-adjunction whose right adjoint {\sf Total} assigns to every split restriction category $\X$ its category of {\em total maps}, structured by its class of {\em restriction isomorphisms}. We show that, when forming the unit of this (very large) adjunction at a ``structured object" $(\C,\S)$, it becomes an isomorphism {\em precisely} when $\S$ is a class of monomorphisms in $\C$. In other words, we present a global 2-adjunction, whose restriction to its ``fixed objects" is exactly the 2-equivalence used in \cite{colackrest} for the characterization of split restriction categories. Likewise, {\sf RaPar} is the left adjoint of a 2-adjunction involving split range categories and categories $\C$ structured by a class $\S$, which now must be part of a factorization system of $\C$ that is stable under pullback along $\S$-morphisms. And again, the 2-equivalence characterizing split range categories as established in \cite{cgh} emerges as the restriction of our 2-adjunction to its fixed objects.

{\em Acknowledgements.} We thank the anonymous referee of the first version of this paper (communicated under a different title as \cite{abandon}) for various questions and suggestions which, in particular, helped us craft the current version of Section 7. We are also grateful to Fernando Lucatelli Nunes who pointed us to the Appendix of Hermida's paper \cite{Hermida}, to emphasize the 2-categorical significance of the $\S$-sectionable and $\S$-retractable span categories, as we indicate in Section 10.

 \section{Span categories and their quotients}
Throughout this paper, we consider a class $\mathcal{S}$ of morphisms in a category $\C$ such that
\begin{itemize}
\item
$\S$ contains all isomorphisms and is closed under composition, and
\item
pullbacks of $\S$-morphisms along arbitrary morphisms exist in $\C$ and belong to $\S$.
\end{itemize}
In particular, we may consider $\mathcal{S}$ as a (non-full) subcategory of $\C$ with the same objects as $\C$. For objects $A,B$ in $\C$, an {\em{$\S$-span}} $(s,f)$ with {\em domain} $A$ and {\em codomain} $B$ is given by a pair of morphisms
\begin{center}
	$\xymatrix{A & D\ar[l]_{s}\ar[r]^{f} & B}$
\end{center}
with $s$ in $\S$ and $f$ in $\C$. These are the objects of the category
$${\sf{Span}}(\C,\S)(A,B)$$
whose morphisms $x:(s,f)\longrightarrow(\tilde{s},\tilde{f})$ are given by $\C$-morphisms $x$ with $\tilde{s}\cdot x=s$ and $\tilde{f}\cdot x=f$, to be composed ``vertically" as in $\C$.
\begin{center}
	$\xymatrix{&& D\ar[lld]_{s}\ar[dd]_{x}\ar[rrd]^{f} && \\
	A && && B\\
&& E\ar[llu]^{\tilde{s}}\ar[rru]_{\tilde{f}} &&}$
\end{center}
Of course, isomorphisms in this category are given by isomorphisms, $x$, in $\C$ making the diagram commute. Notationally we will not distinguish between the pair $(s,f)$ and its isomorphism class in
${\sf{Span}}(\C,\S)(A,B)$.

 The hypotheses on $\S$ guarantee that, when composing $(s,f):A\longrightarrow B$ ``horizontally" with an $\S$-span $(t,g):B\longrightarrow C$ via a (tacitly chosen) pullback $(t',f')$ of $(f,t)$ (see the first diagram in the Introduction),
the composite span $(t,g)\cdot (s,f) := (s\cdot t',g\cdot f')$ is again an $\S$-span. We denote the resulting (ordinary) category\footnote{We remind the reader that $\sf{Span}(\C,\S)$ may, unlike $\C$, fail to have small hom-sets.} of isomorphism classes of $\S$-spans by $$\sf{Span}(\C,\S).$$

 Now we can consider a {\em compatible} relation on $\sf{Span}(\C,\S)$, that is: a relation for $\S$-spans such that
\begin{itemize}
\item only $\S$-spans with the same domain and codomain may be related;
\item vertically isomorphic $\S$-spans are related;
\item horizontal composition from either side preserves the relation.
\end{itemize}
It is a routine exercise, and a fact used frequently in this paper, to show that {\em the least equivalence relation for $\S$-spans generated by a given compatible relation is again compatible}.

 For a compatible equivalence relation $\sim$ we denote the $\sim$-equivalence class of $(s,f)$ by $[s,f]_{\sim}$, or simply by $[s,f]$ when the context makes it clear which relation $\sim$ we are referring to, and we write
$$\sf{Span}_{\sim}(\C,\S)$$
for the resulting quotient category.
We have the pair of functors
$$\Phi_{\sim}=\Phi:\C\longrightarrow\; \sf{Span}_{\sim}(\C,\S)\;\longleftarrow\S^{\rm{op}}:\Psi=\Psi_{\sim}$$
$$(f:D\to B)\longmapsto [1_D,f]\quad\quad [s,1_D] \longleftarrow\!\shortmid (A\leftarrow D:s)$$
which let us decompose every morphism $[s,f]:A\to B$ of $\sf{Span}_{\sim}(\C,\S)$ as 
$$[s,f]=[1_D,f]\cdot[s,1_D]=\Phi f\cdot\Psi s.$$ In fact, as we show next, one can easily characterize the pair by a universal property, using the following terminology.

\begin{definition}\label{BC property}
A pair of functors
$$ F:\C\longrightarrow \mathcal{D}\longleftarrow \S^{\rm{op}}:G$$
into some category $\D$ is said to satisfy the {\em Beck-Chevalley (BC-) property} if
\begin{itemize}
\item $F$ and $G$ coincide on objects, 
\item whenever the square on the left is a pullback diagram in $\C$ with $s\in \S$, then the square on the right commutes:
$$\xymatrix{P\ar[r]^{f'}\ar[d]_{s'}\ar@{}[rd]|{\rm{pb}} & E\ar[d]^{s}\\D\ar[r]_{f} & B}\quad\quad\quad\quad
			\xymatrix{FP\ar@{}[rd]|{}\ar[r]^{ Ff'} & FE\\
			FD\ar[r]_{Ff}\ar[u]^{Gs'} & FB.\ar[u]_{Gs}}			
			$$
\end{itemize}
For a compatible equivalence relation $\sim$ on ${\sf Span}(\C,\S)$, the pair $(F,G)$ is called $\sim$-{\em consistent} if
\begin{itemize}
\item whenever $(s,f)\sim(\tilde{s},\tilde{f})$, then $Ff\cdot Gs=F\tilde{f}\cdot G\tilde{s}$.

\end{itemize}
\end{definition}
Trivially, for a compatible equivalence relation $\sim$, the pair $(\Phi,\Psi)$ is $\sim$-consistent and, obviously, it also satisfies the BC-property, since the following two composite spans coincide:

\begin{center}
	$\xymatrix{&& P\ar[ld]_{s'}\ar[rd]^{f'} && \\
	& D\ar[ld]_{1}\ar[rd]^{f}\ar@{}[rr]|{\rm{pb}} & & E\ar[ld]_{s}\ar[rd]^{1} & \\
D && B && E}$
\hfil
	$\xymatrix{&& P\ar[ld]_{1}\ar[rd]^{1} && \\
	& P\ar[ld]_{s'}\ar[rd]^{1} & & E\ar[ld]_{1}\ar[rd]^{f'} & \\
D && P && E}$
\end{center}
We now confirm that $(\Phi,\Psi)$ is universal with these properties, as follows:

 \begin{proposition}\label{BC}
For a compatible equivalence relation $\sim$ on $\sf{Span}(\C,\S)$, every $\sim$-consistent pair of functors $(F,G)$ satisfying the Beck-Chevalley property factors as $F=H\Phi,\, G=H\Psi$, with a uniquely determined functor $H$, as in
\begin{center}
			$\xymatrix{\C\ar[rr]^{\Phi}\ar[rrd]_{F} && \sf{Span}_{\sim}(\C,\S)			\ar[d]_{H} && \mathcal S^{\rm{op}}\ar[ll]_{\Psi}\ar[lld]^{G}\\
			&& \mathcal D &&}$
		\end{center}
 \end{proposition}

 \begin{proof}
Since every morphism $[s,f]$ in $\sf{Span}_{\sim}(\C,\S)$ factors as $[s,f]=\Phi f\cdot \Psi s$, any functor $H$ factoring as claimed must necessarily map $[s,f]$ to $Ff\cdot Gs$, and $\sim$-consistency allows us to define $H$ in this way. Trivially then, $H$ preserves the identity morphisms [1,1], and the Beck-Chevalley property ensures the preservation of composition:
\begin{center}
 \begin{tabular}{ll}
$H([t,g]\cdot[s,f])=H[s\cdot t',g\cdot f']$&$=F(g\cdot f')\cdot G(s\cdot t')$\\&$=Fg\cdot Ff'\cdot Gt'\cdot Gs$\\&$= Fg\cdot Gt\cdot Ff\cdot Gs$\\&$=H[t,g]\cdot H[s,f].$\\
\end{tabular}
\end{center}
 \end{proof}
 
 \begin{remark}\label{initial}
 Fixing a compatible equivalence relation $\sim$ on $\sf{Span}(\C,\S)$, one can set up the (very large) category ${\sf CBC}(\C,\S,\sim)$, whose objects are triples $(\D,F,G)$ with $(F,G)$ a $\sim$-consistent pair satisfying the Beck-Chevalley property; a morphism $H:(\D,F,G)\to(\D',F',G')$ is simply a functor $H:\D\to\D'$ satisfying $HF=F'$ and $HG=G'$. Then 
 $${\sf CBC}(\C,\S,\sim)\to{\bf CAT},\quad (\D,F,G)\mapsto \D,$$
 defines an opfibration, and {\em Proposition \ref{BC}} describes $(\sf{Span}_{\sim}(\C,\S),\Phi,\Psi)$ as an initial object in ${\sf CBC}(\C,\S,\sim)$. 
  \end{remark}

 \section{The $\S$-retractable span category {\sf Retr}$(\C,\S)$ of $\C$}

 \begin{definition}\label{retractive}
For $\S$-spans with the same domain and codomain we consider the preorder defined by
$$(s,f)\leqslant(\tilde{s},\tilde{f})\Longleftrightarrow \exists\; x:(s,f)\longrightarrow(\tilde{s},\tilde{f}),\;
x\in\S,$$
and call the least equivalence relation on all $\S$-spans containing the reflexive and transitive relation $\leqslant$ the {\em zig-zag relation}, denoting it by $\sim_{\rm{z}}$. The obvious compatibility of $\leqslant$ makes $\sim_{\rm{z}}$ also compatible.
Writing just z instead of $\sim_{\rm{z}}$ when $\sim_{\rm{z}}$ is used as an index, we call the quotient category
$$\sf {Retr(\C,\S)}:=\sf{Span}_{\rm{z}}(\C,\S)$$
the $\S${\em-retractable span category of} $\C$.
It comes with the functors
$$\Phi_{\rm{z}}:\C\longrightarrow\; \sf{Retr}(\C,\S)\;\longleftarrow\S^{\rm{op}}:\Psi_{\rm{z}}$$
$$(f:D\to B)\longmapsto [1_D,f]_{\rm{z}}\quad\quad [s,1_D]_{\rm{z}} \longleftarrow\!\shortmid (A\leftarrow D:s).$$
\end{definition}

The terminology may be justified by the fact that $\Phi_{\rm z}$ transforms $\S$-morphisms into retractions, in a universal manner:

 \begin{proposition}\label{connectivity universal}
For all $s\in \S$ one has $\Phi_{\rm{z}} s \cdot\Psi_{\rm{z}} s=1$. Furthermore,
any pair of functors $(F:\C\to\D,\;G:\S^{\rm op}\to\D)$ satisfying the Beck-Chevalley property and the equalities $Fs\cdot Gs=1\;(s\in\S)$ factors through the pair $(\Phi_{\rm{z}},\Psi_{\rm{z}})$, by a functor $H:{\sf Retr}(\C,\S)\to\D$ that is uniquely determined by $H\Phi_{\rm z}=F,\;H\Psi_{\rm z}=G$.
\end{proposition}

 \begin{proof}
 Since, for $s\in\S$, one trivially has
$(s,s)\leqslant(1,1)$, we see that
$$\Phi_{\rm{z}} s \cdot\Psi_{\rm{z}} s=[1,s]_{\rm{z}}\cdot[s,1]_{\rm{z}}=[s,s]_{\rm{z}}=[1,1]_{\rm{z}}=1.$$
For the stated universal property, applying Proposition \ref{BC}, we just need to confirm that the equalities $Fs\cdot Gs=1\,(s\in\S)$ make $(F,G)$ $\sim_{\rm{z}}$-consistent. But this point is clear since, when
$(s,f)\leqslant(\tilde{s},\tilde{f})$, so that $\tilde{s}\cdot x=s,\,\tilde{f}\cdot x=f$ for some $x\in\S$, we have
$$Ff\cdot Gs=F\tilde{f}\cdot Fx\cdot Gx\cdot G\tilde{s}=F\tilde{f}\cdot G\tilde{s}.$$
\end{proof}

\begin{remark}\label{initial again}
Similarly to {\em Remark \ref{initial}}, one may set up a (very large) category in which $({\sf Retr}(\C,\S),\Phi_{\rm z},\Psi_{\rm z})$ is described as an initial object.
\end{remark}

We may think of Proposition \ref{connectivity universal} as ``going halfway" towards the construction of the {\em category $\C[\S^{-1}]$ of fractions} with respect to $\S$ (see \cite{gabzis, bor}). While we will return to this aspect in Section 5, let us mention immediately the following consequence of Proposition \ref{connectivity universal}:
\begin{corollary}\label{mono imply section}
When $\S$ is a class of monomorphisms in $\C$, then the category ${\sf Retr}(\C,\S)$ is (isomorphic to) the category $\C[\S^{-1}]$. 
\end{corollary}

\begin{proof}
It suffices to show that the functor $\Phi_{\rm z}:\C\to{\sf Retr}(\C,\S)$ is universal with respect to mapping $\S$-morphisms into isomorphisms. This, however, is a straight consequence of Proposition \ref{connectivity universal} once we have proved the following easy Lemma.
\end{proof}

\begin{lemma}\label{BC Mono Lemma}
When $\S$ is a class of monomorphisms in $\C$, any functor $F:\C\to\D$ maps $\S$-morphisms into isomorphisms if, and only if, there is a functor $G:\S^{\rm op}\to\D$ such that $Fs\cdot Gs=1$ for all $s\in\S$, and the pair 
$(F,G)$ satisfies the BC-property; such functor $G$ is then uniquely determined by $F$.
\end{lemma}

\begin{proof}
Clearly, if $F(\S)$ is a class of isomorphisms in $\D$, the functor $G$ defined by $Gs=(Fs)^{-1}\;(s\in\S)$ has the desired properties. Conversely, 
for a monomorphism $s$ the following square on the left is a pullback diagram, so that the Beck-Chevalley property makes the square on the right commute:
$$\xymatrix{D\ar[r]^{1}\ar[d]_{1}\ar@{}[rd]|{\rm{pb}} & D\ar[d]^{s}\\D\ar[r]_{s} & A}\quad\quad\quad\quad
			\xymatrix{FD\ar@{}[rd]|{}\ar[r]^{1} & FD\\
			FD\ar[r]_{Fs}\ar[u]^{1} & FA\ar[u]_{G s}}.$$
			In conjunction with the hypothesis $Fs\cdot Gs=1$, this makes $Fs$ an isomorphism.
\end{proof}

\begin{remark}\label{trivial retractive}
We caution the reader that often the category ${\sf Retr}(\C,\S)$ (and, consequently, also the fraction category $\C[\S^{-1}]$) turns out to be trivial. For example:
 {\em If $\C$ has an initial object $0$ and $\S$ contains the morphisms $!^C:0\to C$ for all objects $C$ in $\C$, then ${\sf Retr}(\C,\S)$ is equivalent to the terminal category $\sf 1$, i.e., all hom-sets of ${\sf Retr}(\C,\S)$ are singletons.}
 Indeed, with the provision $!^A\in\S$ one has $(!^A,!^B)\leqslant(s,f)$, for all $\S$-spans $(s,f):A\to B$. Note that when the morphisms $0\to C,\,C\in\C,$ are monic and, in particular when $0$ is {\em strict initial}, so that the morphisms $C\to 0$ must be isomorphisms, then $!^C$ is a pullback of $0\to 1$, for $1$ terminal in $\C$; hence, having $0\to 1$ in $\S$ suffices to render ${\sf Retr}(\C,S)$ trivial in this case.

 \end{remark}




 \section{The $\S$-sectionable span category {\sf Sect}$(\C,\S)$ of $\C$}
Our next goal is to describe a compatible equivalence relation $\sim$ for $\S$-spans such that (in the notation of Section 2) $\Phi_{\sim}s$ is a section with retraction $\Psi_{\sim}s$, for all $s\in\S$, in a universal manner.
This equivalence relation will be induced by a certain relation for $\S$-{\em cospans}. These are isomorphism classes 
$\langle f, s\rangle$ of pairs 
\begin{center}
	$\xymatrix{A\ar[r]^f & D & B\ar[l]_{s}}$
\end{center}
of $\C$-morphisms with $s\in\S$; $A$ is the {\em domain} and $B$ the {\em codomain} of such an $\S$-cospan. Like for $\S$-spans, isomorphisms of $\S$-cospans live in the category
$${\sf{Cospan}}(\C,\S)(A,B),$$
which has ``vertical" morphisms $v:\langle f,s\rangle\longrightarrow\langle\tilde{f},\tilde{s}\rangle$
obeying $v\cdot f=\tilde{f}, v\cdot s=\tilde{s}$. We call a relation for $\S$-cospans {\em compatible} if
\begin{itemize}
\item only $\S$-cospans with the same domain and codomain may be related;
\item vertically isomorphic $\S$-cospans are related;
\item ``horizontal whiskering" by pre-composition from either side preserves the relation, that is: whenever $\langle f,s\rangle,\langle g,t\rangle$ are related, then also $\langle f\cdot h,s\cdot r\rangle,\langle g\cdot h,t\cdot r\rangle$ are related, for all $\C$-morphisms $h$ and $\S$-morphisms $r$ such that the composites $f\cdot h,\, s\cdot r$ are defined.
\end{itemize}
Like for $\S$-span relations one easily confirms that {\em the least equivalence relation for $\S$-cospans containing a given compatible relation for $\S$-cospans is again compatible}.

 Every $\S$-cospan $\langle f,s\rangle$ gives, via pullback, the $\S$-span $(s',f')={\rm{pb}}\langle f,s\rangle$. In fact, for objects $A,B$ in $\C$ one has a functor
$${\rm{pb}}:{\sf Cospan}(\C,\S)(A,B)\longrightarrow{\sf Span}(\C,\S)(A,B)$$
whose canonical definition on morphisms we will exploit only in Section 6; here the consideration of its action on objects suffices.

 \def\dv{\wr\wr}
\def\ddv{\hbox{\ $\wr\wr$\ }}

 \begin{definition}\label{amalgamated}
{\rm (1)} Like for $\S$-spans, one defines a preorder for $\S$-cospans with the same domain $A$ and codomain $B$ by
$$\langle f,s\rangle\eqslantless \langle\tilde{f},\tilde{s}\rangle\Longleftrightarrow \exists\; v:\langle f,s\rangle\longrightarrow\langle\tilde{f},\tilde{s}\rangle,\;v\in\S.$$
The preorder is obviously a compatible relation for $\S$-cospans.

{\rm (2)} For any compatible $\S$-cospan relation $\dv$, we call the least compatible $\S$-span relation $\approx$ satisfying
$$\langle f,s\rangle\ddv\langle g,t\rangle\Longrightarrow{\rm pb}\langle f,s\rangle\approx{\rm pb}\langle g,t\rangle$$
for all $\S$-cospans $\langle f,s\rangle,\langle g,t\rangle$ the {\em associated} $\S$-span relation of $\dv$.
\end{definition}

Before considering the associated $\S$-span relation of the $\S$-cospan relation $\eqslantless$, we need to describe the association procedure more explicitly:




 \def\dv{\wr\wr}
\def\ddv{\hbox{\ $\wr\wr$\ }}

 \begin{proposition}\label{associated span relation}
Let $\dv$ be a compatible $\S$-cospan relation. Its associated $\S$-span relation $\approx$ may be describeded as follows:
	
$(s,f)\approx(t,g)$ if, and only if, there exist morphisms $u$ in $\S$, $k$ in $\C$, and $\S$-cospans $\langle \check f,\check s\rangle, \langle \check g,\check t\rangle$ such that 
$$\langle \check f, \check s \rangle \ddv \langle \check g,\check t\rangle,\;(s,f)=(1,k)\cdot{\rm pb}\langle\check{f},\check{s}\rangle\cdot(u,1),\;(t,g)=(1,k)\cdot{\rm pb}\langle\check{g},\check{t}\rangle\cdot(u,1);$$		
the latter two identities mean that, for some pullback diagrams

 \begin{center}
$\xymatrix{D\ar[r]^{\hat{f}}\ar[d]_{\hat{s}} & K\ar[d]^{\check s}\\
			U\ar[r]_{\check f} & P}$\quad\quad
		$\xymatrix{E\ar[r]^{\hat{g}}\ar[d]_{\hat{t}} & K\ar[d]^{\check t}\\
			U\ar[r]_{\check g} & Q,}$
	\end{center}

\noindent one obtains the commutative diagram
	\begin{center}
		$\xymatrix{&& D\ar[lld]_{s}\ar[ld]^{\hat{s}}\ar[rd]_{\hat{f}}\ar[rrd]^{f} &&\\
			A &U\ar[l]_{u}&& K\ar[r]^{k} & B\\
			&& E\ar[llu]^{t}\ar[lu]_{\hat{t}}\ar[ru]^{\hat{g}}\ar[rru]_{g} &&}$
	\end{center}	
\end{proposition}

\begin{proof}
Let us denote by $\hat{\approx}$ the $\S$-span relation defined by the description of $\approx$ as claimed by the Proposition, so that our task is show $\approx \,= \hat{\approx}$.
For that it suffices to confirm that $\hat{\approx}$ is compatible, since it trivially satisfies ($\langle f,s\rangle\ddv\langle g,t\rangle\Longrightarrow{\rm pb}\langle f,s\rangle\,\hat{\approx}\,{\rm pb}\langle g,t\rangle$) and is obviously minimal with respect to that property and compatibility.

The $\S$-span relation $\hat{\approx}$ is certainly reflexive, that is: invariant under vertical isomorphism. Indeed, given an $\S$-span $(s,f):A\longrightarrow B$, one has the commutative diagram on the left and the trivial pullback diagram on the right:	
\begin{center}
		$\xymatrix{&& D\ar[lld]_{s}\ar[ld]^{1}\ar[rd]_{f}\ar[rrd]^{f} &&\\
			A &D\ar[l]_{s}&& B\ar[r]^{1} & B\\
			&& D\ar[llu]^{s}\ar[lu]_{1}\ar[ru]^{f}\ar[rru]_{f} &&}$
		\hfil
$\xymatrix{D\ar[r]^{f}\ar[d]_{1} & B\ar[d]^{1}\\
			D\ar[r]_{f} & B}$
	\end{center}
\noindent Hence, with $\langle f,1\rangle\ddv\langle f,1\rangle$ by reflexivity of $\dv$, one concludes $(s,f)\,\hat{\approx}\,(s,f)$.

To prove the invariance of $\hat{\approx}$ under horizontal composition, 
we consider $(s,f)\,\hat{\approx}\,(t,g)$ and first show $(r,h)\cdot(s,f)\,\hat{\approx}\,(r,h)\cdot(t,g)$, for all $(r,h)$ post-composable with $(s,f),\,(t,g)$. By hypothesis, we are given morphisms $\check s,\check t, u$ in $\mathcal{S}$ and $\check f,\check g, k$ in $\C$ such that $\langle\check f,\check s\rangle\ddv\langle \check g,\check t\rangle$ and, for the two pullback diagrams below on the right, the diagram on the left commutes.
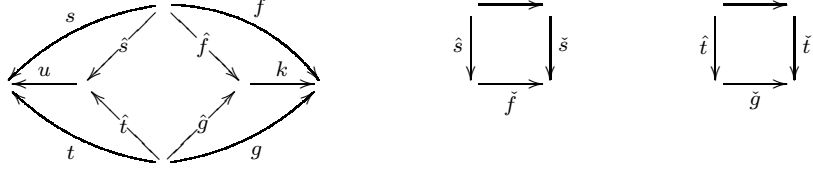
\begin{center}
		$\xymatrix{&& \ar@/_.5pc/[lld]_{s}\ar[ld]|{\hat{s}}\ar[rd]|{\hat{f}}\ar@/^.7pc/[rrd]^{f} &&\\
			&\ar[l]_{u}&& \ar[r]^{k} & \\
			&& \ar@/^.5pc/[llu]^{t}\ar[lu]|{\hat{t}}\ar[ru]|{\hat{g}}\ar@/_.5pc/[rru]_{g} &&}$
		\hfil
		$\xymatrix{\ar[r]^{\hat{f}}\ar[d]_{\hat{s}} & \ar[d]^{\check s}\\
			\ar[r]_{\check f} &}$
		\hfil
		$\xymatrix{\ar[r]^{\hat{g}}\ar[d]_{\hat{t}} & \ar[d]^{\check t}\\
			\ar[r]_{\check g} & }$
	\end{center}
\noindent	The equalities $k\cdot \hat{f}=f$ and $k\cdot \hat{g}=g$ produce the following commutative diagrams, in which the squares are pullbacks (here $x^*(y)$ denotes a pullback of $y$ along $x$):
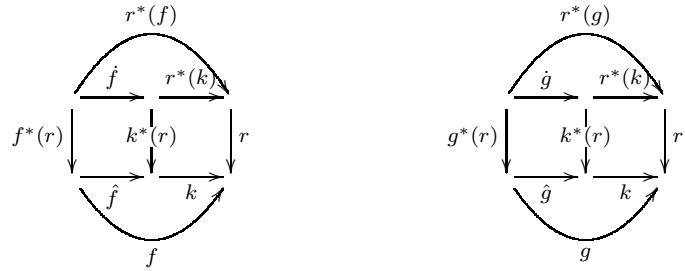
\begin{center}
		$\xymatrix{\ar[r]^{\dot{f}}\ar@/^2pc/[rr]^{r^{*}(f)}\ar[d]_{f^{*}(r)} & \ar[d]|{k^{*}(r)}\ar[r]^{r^{*}(k)} & \ar[d]^{r}\\
			\ar[r]_{\hat{f}}\ar@/_2pc/[rr]_{f} & \ar[r]_{k} &}$
		\hfil
		$\xymatrix{\ar[r]^{\dot{g}}\ar@/^2pc/[rr]^{r^{*}(g)}\ar[d]_{g^{*}(r)} & \ar[d]|{k^{*}(r)}\ar[r]^{r^{*}(k)} & \ar[d]^{r}\\
			\ar[r]_{\hat{g}}\ar@/_2pc/[rr]_{g} & \ar[r]_{k} &}$
	\end{center}	
	With the pullback diagrams on the right, it is easy to see that the following diagram on the left commutes:
	\begin{center}
		$
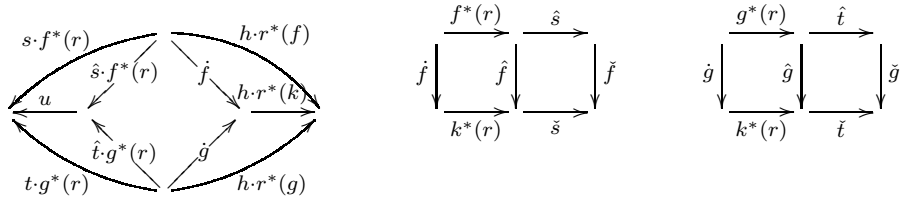
\xymatrix{&& \ar@/_.5pc/[lld]_{s\cdot f^{*}(r)}\ar[ld]|{\hat{s}\cdot f^{*}(r)}\ar[rd]|{\dot{f}}\ar@/^.7pc/[rrd]^{h\cdot r^{*}(f)} &&\\
			&\ar[l]_{u}&& \ar[r]^{h\cdot r^{*}(k)\ \ } & \\
			&& \ar@/^.5pc/[llu]^{t\cdot g^{*}(r)}\ar[lu]|{\hat{t}\cdot g^{*}(r)}\ar[ru]|{\dot{g}}\ar@/_.5pc/[rru]_{h\cdot r^{*}(g)} &&}$
		\hfil
		$\xymatrix{\ar[r]^{f^{*}(r)}\ar[d]_{\dot{f}} & \ar[d]_{\hat{f}}\ar[r]^{\hat{s}} & \ar[d]^{\check f}\\
			\ar[r]_{k^{*}(r)} & \ar[r]_{\check s} &}$
		\hfil
		$\xymatrix{\ar[r]^{g^{*}(r)}\ar[d]_{\dot{g}} & \ar[d]_{\hat{g}}\ar[r]^{\hat{t}} & \ar[d]^{\check g}\\
			\ar[r]_{k^{*}(r)} & \ar[r]_{\check t} &}$
	\end{center}
\noindent	From $\langle \check f,\check s\rangle \ddv\langle \check g,\check t\rangle$, using invariance of $\dv$ under horizontal whiskering, we now obtain $\langle \check f,\check s\cdot k^{*}(r)\rangle\ddv\langle \check g,\check t\cdot k^{*}(r)\rangle$. So, $$(s\cdot f^{*}(r),h\cdot r^{*}(f))\,\hat{\approx}\,(t\cdot g^{*}(r), h\cdot r^{*}(g)),$$ as desired.

 	The proof that $\hat{\approx}$ is also preserved by pre-composition (rather than post-composition) in $\sf{Span(\C,\S)}$ proceeds very similarly.
	\end{proof}

\begin{definition}\label{weak partial} 
We denote the associated $\S$-span relation of the $\S$-cospan relation $\eqslantless$ (of {\em Definition \ref{amalgamated}(1)}) by $\approx_{\rm a}$ and let $\sim_{\rm a}$ denote the equivalence relation generated by $\approx_{\rm a}$; it is given by the symmetric and transitive hull of the compatible relation $\approx_{\rm a}a$, and $\sim_{\rm{a}}$ is therefore compatible as well. Writing simply {\rm a} when $\sim_{\rm{a}}$ is used as an index, we call the quotient category
$$\sf {Sect(\C,\S)}:=\sf{Span}_{\rm{a}}(\C,\S)$$
the $\S${\em-sectionable span category of} $\C$.
It comes with the functors
$$\Phi_{\rm{a}}:\C\longrightarrow\; \sf{Sect}(\C,\S)\;\longleftarrow\S^{\rm{op}}:\Psi_{\rm{a}}$$
$$(f:D\to B)\longmapsto [1_D,f]_{\rm{a}}\quad\quad [s,1_D]_{\rm{a}} \longleftarrow\!\shortmid (A\leftarrow D:s).$$
\end{definition}
Here is the first key property of these functors:

 \begin{lemma}\label{section}
$\Psi_{\rm{a}}s\cdot \Phi_{\rm{a}} s=1$, for every $s\in\S$.
\end{lemma}

 \begin{proof}
Trivially $\langle 1,1\rangle\eqslantless\langle s,s\rangle$.
Consequently, for the kernel pair $(u,v)$ of $s$, $(u,v)\approx_{\rm{a}}(1,1)$ follows, so that $\Phi_{\rm{a}} v\cdot \Psi_{\rm{a}} u=1$. Since, by the Beck-Chevalley property,
$\Psi_{\rm{a}} s\cdot \Phi_{\rm{a}} s=\Phi_{\rm{a}} v\cdot \Psi_{\rm{a}} u$, this completes the proof.
\end{proof}

 We can now prove that $(\Phi_{\rm{a}}, \Psi_{\rm{a}})$ is universal with respect to the identity shown in Lemma \ref{section}:

 \begin{theorem}\label{weak partial universal}
Any pair of functors $F:\C\longrightarrow\mathcal{D},\; G:\S^{\rm{op}}\longrightarrow\mathcal{D}$ satisfying the Beck-Chevalley property and the identity $Gs\cdot Fs=1\;(s\in\S)$ factors as $F=H\Phi_{\rm a},\, G=H\Psi_{\rm a}$, with a uniquely determined functor $H$, as in
\begin{center}
			$\xymatrix{\C\ar[rr]^{\Phi_{\rm a}}\ar[rrd]_{F} && \sf{Sect}(\C,\S)			\ar[d]_{H} && \mathcal S^{\rm{op}}\ar[ll]_{\Psi_{\rm a}}\ar[lld]^{G}\\
			&& \mathcal D &&}$
\end{center}
\end{theorem}

 \begin{proof} After Proposition \ref{BC} it suffices to show that the pair $(F,G)$ is necessarily $\sim_{\rm{a}}$-consistent. Hence we consider $(s,f)\approx_{\rm{a}}(g,t)$ and obtain (as in the proof of Proposition \ref{associated span relation}) the set of commutative diagrams
\begin{center}
		$\xymatrix{&& \ar@/_.5pc/[lld]_{s}\ar[ld]|{\hat{s}}\ar[rd]|{\hat{f}}\ar@/^.7pc/[rrd]^{f} &&\\
			&\ar[l]_{u}&& \ar[r]^{k} & \\
			&& \ar@/^.5pc/[llu]^{t}\ar[lu]|{\hat{t}}\ar[ru]|{\hat{g}}\ar@/_.5pc/[rru]_{g} &&}$
		\hfil
		$\xymatrix{\ar[r]^{\hat{f}}\ar[d]_{\hat{s}} & \ar[d]^{\check s}\\
			\ar[r]_{\check f} &}$
		\hfil
		$\xymatrix{\ar[r]^{\hat{g}}\ar[d]_{\hat{t}} & \ar[d]^{\check t}\\
			\ar[r]_{\check g} & }$
	\end{center}

 \noindent where now $\langle\check{f},\check{s}\rangle \eqslantless \langle\check{g},\check{t}\rangle$. This gives us, in addition, a commutative diagram
\begin{center}
		$\xymatrix{& \ar[dd]_{v} & \\
			\ar[ru]^{\check{f}}\ar[rd]_{\check{g}} & &\ar[lu]_{\check{s}}\ar[ld]^{\check{t}}\\
			& &}$
	\end{center}
	with $v\in\S$. By hypothesis then, $Gv\cdot Fv = 1$. Furthermore,
the above pullback squares and the Beck-Chevalley property give us $F\hat f\cdot G\hat s=G\check s\cdot F\check f$ and $F\hat g\cdot G\hat t=G\check t\cdot F\check g$.
Applying $F$ to $v\cdot\check{f}=\check{g}$ and $G$ to $v\cdot\check{s}=\check{t}$ we then obtain

 \begin{tabular}{ll}
	$Ff\cdot Gs$&$=Fk\cdot F\hat f\cdot G\hat s\cdot Gu=Fk\cdot G\check s\cdot F\check f\cdot Gu$\\
	&$=Fk\cdot G\check s\cdot Gv\cdot Fv\cdot F\check f\cdot Gu=Fk\cdot G\check t \cdot F\check g\cdot Gu$\\
	&$=Fk\cdot F\hat g\cdot G\hat t\cdot Gu$\\&$=Fg\cdot Gt$.\\
\end{tabular}	

 \end{proof}
Let us note immediately that the considerable effort in forming the relation $\sim_{\rm{a}}$ pays off only when $\S$ is not restricted to containing only monomorphisms of $\C$. Indeed, otherwise our construction returns just the category ${\sf Span(\C,\S)}$, studied as the $\S${\em -partial map category} of $\C$ by various authors (see, for example, \cite{homi}):

 \begin{corollary} \label{WPar in mono case}
When $\S$ is a class of monomorphisms,
$(s,f)\sim_{\rm{a}}(t,g)$ just means that the two $\S$-spans are isomorphic. In other words, if $\S$ contains only monomorphisms of $\C$, then $\sf{Sect}(\C,\S)\cong\sf{Span}(\C,\S)$ is (isomorphic to) the $\S$-partial map category of $\C$.
\end{corollary}

 \begin{proof}
 For the pair $(\Phi:\C\longrightarrow\; {\sf Span}(\C,\S)\;\longleftarrow\S^{\rm op}:\Psi)$ one trivially has $\Psi s\cdot\Phi s=1$ for every monomorphism $s\in\S$,
 so that Theorem \ref{weak partial universal} gives us the functor
$$\sf{Sect}(\C,\S)\longrightarrow\sf{Span}(\C,\S),\quad[s,f]_{\rm{a}}\mapsto[s,f]_{\cong}=(s,f),$$
which is trivially inverse to $[s,f]_{\cong}\mapsto [s,f]_{\rm{a}}$.
\end{proof}

 \begin{remark}\label{Par trivial}
It is to be expected that the largest class $\S$ possible, namely $\S={\rm Mor}(\C)$, will render ${\sf Sect}(\C,\S)$ trivial. Concretely, it is easy to see that, similar to {\em Remark \ref{trivial retractive}}, one has:
{\em If $\C$ has disjoint finite coproducts (so that the pullback of two distinct coproduct injections is given by the initial object), then ${\sf Sect}(\C,{\rm Mor}(\C))$ is equivalent to the terminal category.}
In fact, since for all spans $(s,f):A\to B$ one has $\langle\nu_1,\nu_2\rangle\eqslantless\langle f,1_B\rangle$ (with coproduct injections $\nu_1,\nu_2$), the following diagrams show $[s,f]_{\rm a}=[!^A,!^B]_{\rm a}$:
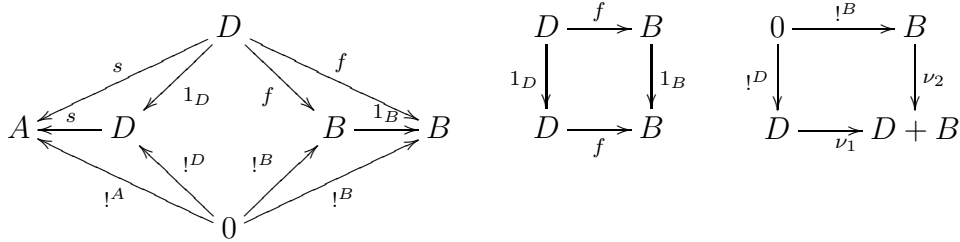
\begin{center}
		$\xymatrix{&& D\ar[lld]_{s}\ar[ld]^{1_D}\ar[rd]_{f}\ar[rrd]^{f} &&\\
			A &D\ar[l]_{s}&& B\ar[r]^{1_B} & B\\
			&& 0\ar[llu]^{!^A}\ar[lu]_{!^D}\ar[ru]^{!^B}\ar[rru]_{!^B} &&}$
		\hfil
		$\xymatrix{D\ar[r]^{f}\ar[d]_{1_D} & B\ar[d]^{1_B}\\
			D\ar[r]_{f} & B}$
		\hfil
		$\xymatrix{0\ar[r]^{!^B}\ar[d]_{!^D} & B\ar[d]^{\nu_2}\\
			D\ar[r]_{\nu_1} & {D+B}}$	
	\end{center}
\end{remark}

To illustrate the construction of ${\sf Sect}(\C,\S)$ further, we end this section by showing that also for certain quite small classes $\S$ will ${\sf Sect}(\C,\S)$ be trivial, provided that $\S$ contains some morphisms that are typically {\em epimorphisms} in $\C$. To this end we call the category $\C$ {\em strictly connected} if for all objects $A$ the hom-functor $\C(A,-):\C\to{\sf Set}$ reflects strict initial objects (see Remark \ref{trivial retractive}). As $\emptyset$ is strict initial in $\sf Set$, this means that, for all objects $A,B$ the hom-set $\C(A,B)$ may be empty only if $B$ is strict initial in $\C$. (Note however, that we do not impose an existence assumption for a strict initial object in $\C$ when $\C$ is strictly connected!) Every pointed category is trivially strictly connected, but also non-pointed categories like $\sf Set, Ord, Cat, Top, ..., $ (all with strict initial object $\emptyset$) are strictly connected.

 \begin{theorem}\label{Par is 1}
Let $\C$ have a terminal object $1$ and be strictly connected, and let the class $S$ contain
the morphisms $!_A:A\to1$, for all objects $A$ that are not strict initial.
Then all hom-sets of ${\sf Sect}(\C,S)$ contain only either one or two morphisms; they are all singletons when $\C$ has no strict initial object, in which case ${\sf Sect}(\C,S)$
is equivalent to the terminal category $\sf1$.
\end{theorem}
\begin{proof}
Consider an $\S$-span $(s:D\to A,\; f:D\to B)$. If $B$ is strict initial, $f$ is necessarily an isomorphism, thus making $D\cong 0$ strict initial as well. Hence, the $S$-spans $(s,f)$ and $(!^A, !^B)$ coincide. If $B$ is not strict initial, we have $!_B\in\S$, so its pullback along $!_D$ exists and gives us the direct product $D\times B$, with projections $p_1, p_2$, where $p_1:D\times B\to D$ is in $\S$ by pullback stability of $\S$. The $\S$-cospan inequality
$\langle f,1_B\rangle\eqslantless\langle!_D,!_B\rangle$ and the commutative diagrams

 \begin{center}
$
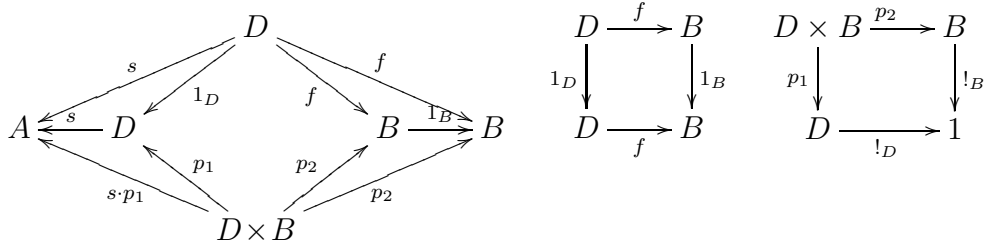
\xymatrix{&& D\ar[lld]_{s}\ar[ld]^{1_D}\ar[rd]_{f}\ar[rrd]^{f} &&\\
			A &D\ar[l]_{s}&& B\ar[r]^{1_B} & B\\
			&& {D\!\times\! B}\ar[llu]^{s\cdot p_1}\ar[lu]_{p_1}\ar[ru]^{p_2}\ar[rru]_{p_2} &&}$
		\hfil
		$\xymatrix{D\ar[r]^{f}\ar[d]_{1_D} & B\ar[d]^{1_B}\\
			D\ar[r]_{f} & B}$
		\hfil
		$\xymatrix{{D\times B}\ar[r]^{p_2}\ar[d]_{p_1} & B\ar[d]^{!_B}\\
			D\ar[r]_{!_D} & {1}}$
\end{center}			
			
\noindent show $[s,f]_{\rm a}=[s\cdot p_1,p_2]_{\rm a}$. Hence, it suffices to consider the $\S$-span $(s\cdot p_1, p_2)$, with $B$ not strict initial.

 If $D\times B\cong0$ is strict initial, the $S$-span $(s\cdot p_1, p_2)$ must equal $(!^A,!^B)$. If $D\times B$ is not strict initial, $D$ cannot be strict initial either, and we have a morphism $a:A\to D$, by the strict connectedness of $\C$. Also, just as the product $D\times B$ exists, so does the product $A\times B$, with product projections $\pi_1, \pi_2$, where $\pi_1\in\S$, and we can consider the diagrams

 \begin{center}
$
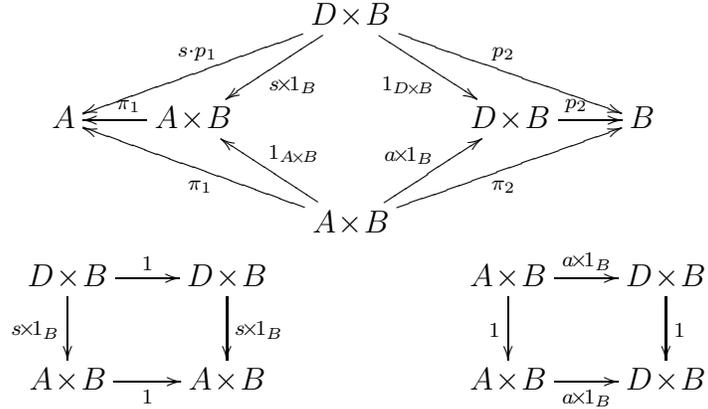
\xymatrix{&& {D\!\times\! B}\ar[lld]_{s\cdot p_1}\ar[ld]^{s\!\times\!1_B}\ar[rd]_{1_{D\!\times\!B}}\ar[rrd]^{p_2} &&\\
			A &{A\!\times\! B}\ar[l]_{\pi_1}&& {D\!\times \!B}\ar[r]^{p_2} & B\\
			&& {A\!\times \!B}\ar[llu]^{\pi_1}\ar[lu]_{1_{A\!\times\!B}}\ar[ru]^{a\!\times\!1_B}\ar[rru]_{\pi_2} &&}$
		\hfil
		
		$\xymatrix{D\!\times\!B\ar[r]^{1}\ar[d]_{s\!\times\!1_B} & {D\!\times\!B}\ar[d]^{s\!\times\!1_B}\\
			{A\!\times\!B}\ar[r]_{1} & {A\!\times\!B}}$
		\hfil
		$\xymatrix{{A\!\times\! B}\ar[r]^{a\!\times\!1_B}\ar[d]_{1} & {D\!\times\!B}\ar[d]^{1}\\
			{A\!\times\!B}\ar[r]_{a\!\times\!1_B} & {D\!\times\!B}}$	
\end{center}

 \noindent The morphism $s\!\times\!1_B$ shows that $A\times B$ is, like $D\times B$, not strict initial, so that the morphisms $!_{A\!\times\!B},\,!_{D\!\times\!B}$ both lie in $\S$. This gives the $\S$-cospan inequalities $$\langle1_{A\!\times\!B},s\!\times\!1_B\rangle\eqslantless\langle!_{A\!\times\!B},!_{A\!\times\!B}\rangle\eqslantgtr\langle a\!\times\!1_B,!_{D\!\times\!B}\rangle,$$ which then imply $[s\cdot p_1,p_2]_{\rm a}=[\pi_1,\pi_2]_{\rm a}$.

 In summary: when $B$ is strict initial in $\C$, $[!^A,!^B]_{\rm a}$ is the only morphism $A\to B$ in ${\sf Sect}(\C,S)$; otherwise one may also have the morphism $[\pi_1,\pi_2]_{\rm a}$, but no other.
\end{proof}

 \begin{remark}\label{rem on Par}
{\em (1)} In every category $\C$ with finite products there is a least class $\S$ which satisfies our general hypotheses and contains the morphisms $A\to 1$, for every object $A$ that is not strict initial in $\C$; namely, the class ${\sf Proj}(\C)$ of all morphisms that are either projections of a direct product that is not strict initial, or that are isomorphisms of strict initial objects. Hence, when $\C$ is strictly connected, the assertion of Theorem \ref{Par is 1} applies for $\S={\sf Proj}(\C)$.

 {\em (2) Theorem \ref{Par is 1}} leaves open the question whether, when $\C$ is strictly connected and has finite products and a strictly initial object $0$, the ${\sf Sect}(\C,\S)$-morphisms $$0_{A,B}:=[!^A,!^B]_{\rm a},\,1_{A,B}:=[\pi_1,\pi_2]_{\rm a}:A\to B$$ are actually distinct. For $\S={\sf Proj}(\C)$ it is not difficult to show that, if every object in $C$ is projective with respect to ${\sf Proj}(\C)$, then
$$0_{A,B}=1_{A,B}\Longleftrightarrow A\times B \,{\rm \;strict\; initial}.$$
In particular, for $\C={\sf Set}$ and $\S={\sf Proj(Set)}$, one has $0_{A,B}\neq1_{A,B}$ for all non-empty sets $A,B$.
\end{remark}

 \section{The category $\C[\S^{-1}]$ of fractions}
It is now easy to construct the category $\C[\S^{-1}]$ of fractions with respect to the class $\S$ satisfying our general hypotheses (but {\em not} necessarily being a class of monomorphisms of $\C$) as a quotient category of $\sf{Sect}(\C,\S)$. Recall (\cite{gabzis, bor}) that the category $\C[\S^{-1}]$
is characterized by the admission of a {\em localizing functor} $\C\longrightarrow\C[\S^{-1}]$, universal with the property that it maps morphisms in $\S$ to isomorphisms.

 In order to construct such localizing functor we consider the least equivalence relation $\sim_{\rm{az}}$ for $\S$-spans containing both the zig-zag relation $\sim_{\rm{z}}$ of Definition \ref{retractive} and the associated relation $\sim_{\rm{a}}$ of Definition \ref{weak partial}.
 As both generating relations are compatible, the relation $\sim_{\rm{az}}$ is compatible as well, and we can consider the pair of functors
$$\Phi_{\rm{az}}:\C\longrightarrow\; \sf{Span}_{\rm{az}}(\C,\S)\;\longleftarrow\S^{\rm{op}}:\Psi_{\rm{az}}\quad$$
(defined as in Proposition \ref{BC}) which, by the definition of $\sim_{\rm{az}}$, factors through both $\sf{Span}_{\rm{a}}(\C,\S)$ and $\sf{Span}_{\rm{z}}(\C,\S)$. For all $s\in\S$, this makes $\Phi_{\rm{az}} s$, by Lemma \ref{section} and Proposition \ref{connectivity universal}, both a section and a retraction, and therefore an isomorphism, with $\Psi_{\rm{az}} s$ being its inverse.

 \begin{theorem}\label{fraction}
$ \sf{Span}_{\rm{az}}(\C,\S)$ is (a model of) the category $\C[\S^{-1}]$ of fractions with respect to the class $\S$, with localizing functor $\Phi_{\rm{az}}$.
\end{theorem}

 \begin{proof}
It just remains to be shown that any functor $F:\C\longrightarrow\mathcal{D}$ which maps every $s\in\S$ to an isomorphism factors uniquely through $\Phi_{\rm{az}}$.
By Theorem \ref{weak partial universal} and Corollary \ref{connectivity universal}, with
$G:\S^{\rm{op}}\longrightarrow \mathcal D, \; s\mapsto (Fs)^{-1},$
we obtain a pair $(F,G)$ that is both $\sim_{\rm{a}}$- and $\sim_{\rm{z}}$-consistent and therefore also $\sim_{\rm{az}}$-consistent. Since it trivially satisfies the Beck-Chevalley property, Proposition \ref{BC} produces the unique factorization of $F$ through $\Phi_{\rm{az}}$, given by
$${\sf{Span}}_{\rm{az}}(\C,\S)\longrightarrow \mathcal D,\quad [s,f]_{\rm{az}}\mapsto Ff\cdot (Fs)^{-1}.$$
\end{proof}

As a consequence, whenever ${\sf Retr}(\C,\S)$ or ${\sf Sect}(\C,\S)$ is trivial, {\em i.e,} equivalent to {\sf 1}, so is $\C[\S^{-1}]$. In particular, from Theorem \ref{Par is 1} we conclude: 

 \begin{corollary}\label{trivial fraction}
If $\C$ is strictly connected and contains a terminal object $1$, but does not contain a strict initial object, and if $\S$ contains all morphisms with codomain $1$, then $\C[\S^{-1}]$ is equivalent to the terminal category $\sf 1$.
\end{corollary}

Here is an easy example of a class $\S$ satisfying the preset general hypotheses, but not contained in the class of monomorphisms of its parent category $\C$ and not 
trivializing $\C[\S^{-1}]$:

 \begin{example}\label{Ord}
In the category $\sf Ord$ of preordered sets and their monotone (= order-preserving) maps, let $\S$ be the class of fully faithful surjections $f:X\to Y$, {\em i.e.}, of surjective maps $f$ with $(x\leq x'\iff f(x)\leq f(x'))$ for all $x,x'\in X$. Note that such maps are special equivalences of preordered sets, these being considered as small ``thin" categories. We claim that ${\sf Ord}[\S^{-1}]$ {\em is equivalent to the full reflective subcategory $\sf Pos$ of $\sf Ord$ of partially ordered sets} and first show that the reflector $P:{\sf Ord}\to{\sf Pos}$ maps morphisms in $\S$ to isomorphisms.

 Indeed, with the Axiom of Choice granted, its surjectivity makes every $s:X\to Y$ in $\S$ have a section $s'$ in $\sf Set$ which, since $s$ is fully faithful, actually lives in $\sf Ord$. Writing $(x\simeq \tilde{x}\iff x\leq \tilde{x} \;{\rm and}\; \tilde{x}\leq x)$ for all $x,\tilde{x}\in X$, so that the reflection of $X$ into $\sf Pos$ may be taken to be the projection $p_X:X\to X/\!\!\simeq\;=PX$, from $s'(s(x))\simeq x$ and $s(s'(y))=y$ for all $x\in X, y\in Y$ we conclude that $Ps'$ is inverse to $Ps$ in $\sf Pos$. Consequently, $P$ factors uniquely through $\Phi_{\rm az}$, by the functor
$$\bar{P}:{\sf Ord}[\S^{-1}]\to{\sf Pos},\;[s,f]_{\rm az}\mapsto Pf\cdot (Ps)^{-1}=P(f\cdot s').$$
We show that $\bar{P}$ is an equivalence of categories. Certainly, $\bar{P}$ is, like $P$, essentially surjective on objects. Noting that the reflection maps belong to $\S$, for any monotone map $h:PX\to PY$ we have the monotone map $g:=(p_Y)'\cdot h\cdot p_X:X\to Y$, so that $Pg\cdot p_X=p_Y\cdot g=p_Y\cdot (p_Y)'\cdot h\cdot p_X=h\cdot p_X$ and then
$\bar{P}([1_X,g]_{\rm az})=Pg=h$ follows. Suppose that also $[s,f]_{\rm az}:X\to Y$ satisfies $\bar{P}([s,f]_{\rm az})=h$, so that $P(f\cdot s')=Pg$ and then $p_Y\cdot(f\cdot s')=p_Y\cdot g$. With $p_Y\in\S$ one obtains the $\S$-cospan inequalities
$$\langle f\cdot s',1_Y\rangle\eqslantless\langle p_Y\cdot (f\cdot s'),p_Y\rangle=\langle p_Y\cdot g,p_Y\rangle\eqslantgtr\langle g,1_Y\rangle,$$
which imply $[s,f]_{\rm a}=[1_X,f\cdot s']_{\rm a}=[1_X,g]_{\rm a}$ and then $[s,f]_{\rm az}=[1_X,g]_{\rm az}.$
This shows that $\bar{P}$ is fully faithful.
\end{example}

 \begin{remark}\label{existence} As a quotient category of $\sf{Span}(\C,\S)$, in general the category $\C[\S^{-1}]$ may still fail to have small hom-sets. In fact, only few handy criteria are known that would guarantee its hom-sets to be small when $\C$ has small hom-sets. One such criterion is the following (see, for example, {\em  \cite{schubert}, Theorem 19.3.1}): With $\C$ finitely complete, let $\S$ be the class of morphisms mapped to isomorphisms by some functor $S:\C\longrightarrow \mathcal{B}$ which
preserves finite limits. If $\S$ admits a so-called {\em calculus of right fractions}, then the hom-sets of $\C[\S^{-1}]$ are small. Moreover, the factorizing functor $\bar{S}:\C[\S^{-1}]\longrightarrow \mathcal{B}$ with $\bar{S}\Phi_{\rm az}=S$ will not only be conservative ({\em i.e}, reflect isomorphisms), but also preserve finite limits (and, hence, be faithful).
\end{remark}

 \section{The split restriction category {\sf Par}$(\C,\S)$ of $\S$-partial maps in $\C$}

 Cockett and Lack \cite{colackrest} show that the 2-category of categories $\C$ equipped with a class $\S$ of {\em mono}morphisms in $\C$ satisfying our general hypotheses (together with functors and natural transformations compatible with the classes $\S$) is 2-equivalent to the category of so-called split restriction categories (with functors and natural transformations compatible with the restriction structure). The 2-equivalence is furnished by $(\C,\S)\mapsto \sf{Sect}(\C,\S)$ which, when $\S$ contains only monomorphisms, is the category ordinarily known as the category of $\S$-partial maps in $\C$ (see Corollary \ref{WPar in mono case}). However, {\em without} the mono constraint on $\S$,
while $\sf{Sect}(\C,\S)$ is
characterized by the universal property given in Theorem \ref{weak partial universal}, it remains unknown whether the category is a (split) restriction category (see Remark \ref{open problem} below); we suspect that it generally fails to be. Our goal is therefore to find a sufficiently large quotient category 
of ${\sf{Sect}}(\C,\S)$ that is a (split) restriction category and can justifiably take on the role of the $\S$-partial map category of $\C$ in the general (non-mono) case. For subsequent reference, let us first recall the notion of (split) restriction category in detail (see \cite{colackrest}):

 \begin{definition} 
	A {\em restriction structure} on a category is an assignment
\begin{center}
	 $f:A\longrightarrow B$
	
	 $\overline{\bar f:A\longrightarrow A}$
	 \end{center}	
		of a morphism $\bar f$ to each morphism $f$, satisfying the following four conditions:
\begin{itemize} 	
\item[\em{(R1)}] $f\cdot\bar f=f$ for all morphisms $f$;
\item[\em{(R2)}] $\bar f \cdot \bar g=\bar g\cdot\bar f$ whenever ${\rm{dom}}f={\rm{dom}}g$;
\item[\em{(R3)}] $\overline{g\cdot \bar f}=\bar g\cdot\bar f$ whenever ${\rm{dom}}f={\rm{dom}}g$;
\item[\em{(R4)}] $\bar g\cdot f=f\cdot\overline{g\cdot f}$ whenever ${\rm{cod}}f={\rm{dom}}g$.
\end{itemize}
A category with a restriction structure is called a {\em restriction category}.
A morphism $e$ such that $\bar e=e$ is called a {\em restriction idempotent}.\footnote{For all morphisms $f$, $\bar{f}$ is a restriction idempotent: consider $g=1$ in (R3) and use $\bar{1}=1$, from (R1).} A restriction idempotent
	$e$ is said to be {\em split}, if there are morphisms $m$ and $r$ such that $mr=e$ and $rm=1$.
	One says that a restriction structure on a category is {\em split} if all the restriction idempotents are split.
\end{definition}

\begin{remark}\label{open problem}
Mimicking the definition of the restriction structure on ${\sf Span}(\C,\S)$ when $\S$ is a class of monomorphisms in $\C$, one is tempted trying to define the same on ${\sf Sect}(\C,\S)$ in the general case by $\overline{[s,f]}_{{\rm{a}}}:=[s,s]_{{\rm{a}}}$ for all $\S$-spans $(s,f)$. In fact, whilst the computations used in the proof of {\em Theorem \ref{Par split rest}} below show that $\overline{(-)}$ satisfies conditions {\em (R1-4)}, we are not able to confirm that $\overline{(-)}$ is well-defined.
\end{remark}

\begin{open problem}
Does ${\sf Sect}(\C,\S)$ carry a (split) restriction structure which, when $\S$ is a class of monomorphisms, coincides with the standard structure on ${\sf Span}(\C,\S)$?
\end{open problem}

In order to find a suitable quotient category of ${\sf Sect}(\C,\S)$ as indicated above, we return to the functor
$${\rm{pb}}:{\sf Cospan}(\C,\S)(A,B)\longrightarrow{\sf Span}(\C,\S)(A,B)$$
used in Definition \ref{amalgamated} which, for objects $A,B$ in $\C$,
assigns to every $v:\langle f,s\rangle\longrightarrow\langle g,t\rangle$ of $\S$-cospans the canonical morphism $$v^{\star}:(s',f')={\rm pb}\langle f,s\rangle\longrightarrow(t',g')={\rm pb}\langle g,t\rangle,$$ {\em i.e.,} the unique $\C$-morphism $v^{\star}$ rendering commutative the diagram
 \begin{center}
	$
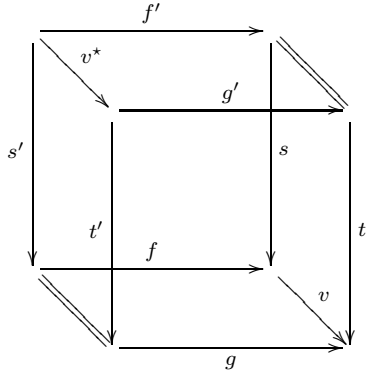
\xymatrix{\ar[rrr]^{f'}\ar[ddd]_{s'}\ar[rd]^{v^{\star}} &&& \ar@{=}[rd]^{}\ar[ddd]^{s} &\\
		& \ar[rrr]^{g'}\ar[ddd]_{t'} &&&\ar[ddd]^{t}\\
		&&&&\\
		\ar[rrr]^{f}\ar@{=}[rd]_{} &&&\ar[rd]^{v} &\\
		&\ar[rrr]_{g} &&&}$
\end{center}

\begin{definition}
We call the $\S$-span morphism  $v^{\star}$ just defined the {\em conjugate} of the $\S$-cospan morphism $v$ and form the class 
\begin{center}
$\S^{\star}:=$ {\em closure under pullback of} $\{v^{\star}\,|\,v\; \S$-{\em cospan morphism}, $v\in\S\}$
\end{center}
of all (existing) pullbacks in $\C$ of the conjugates $v^{\star}$ of all $\S$-cospan morphisms $v$ with $v\in \S$.
\end{definition}

We state some easily verified properties of the class $\S^{\star}$:
\begin{proposition}\label{properties}
{\rm(1)} $\S^{\star}$ contains all isomorphisms and is stable under pullback in $\C$.

{\rm (2)} If $\S$ satisfies the {\em weak left cancellation condition}, so that $s\cdot t\in\S$ with $s\in\S$ implies $t\in\S$, then $\S^{\star}\subseteq\S$.

{\rm (3)} If $\S$ is a class of monomorphisms in $\C$, then $\S^{\star}$ is precisely the class of isomorphisms in $\C$.
\end{proposition}

\begin{proof}
(1) is obvious, and (2) follows from the commutativity of the left panel of the cubic diagram above. For (3), when the morphism $v$ in that diagram is a monomorphism,
the pullback $(s',f')$ of $(f,s)$ serves also as a pullback for $(g,t)$, so that $v^{\star}$ must be an isomorphism.
\end{proof}

\begin{remark}\label{left cancel}
{\em (1)} In general, $\S^{\star}$ is not comparable with $\S$ via inclusion. Indeed, for  $\S$ as in {\em Example \ref{Ord}} one easily shows that $\S^{\star}$ is a class of full embeddings in {\sf Ord}, so that $\S\cap\S^{\star}$ is the class of isomorphisms; consequently, since both $\S$ and $\S^{\star}$ contain not just isomorphisms, this implies that one has neither $\S^{\star}\subseteq\S$ nor $\S\subseteq\S^{\star}$. In particular, the weak left cancellation condition in {\em (2)} of {\em Proposition \ref{properties}} is essential.

{\em (2)} If $\S$ satisfies the weak left cancellation condition,
then, for our purposes,
 it suffices to define $\S^{\star}$ as the closure of the class $\{v^{\star} \,|\,v\;\S$-cospan morphism,\,$v\in\S\,\}$ under pullback just along $\S$-morphisms, not necessarily along all morphisms.
\end{remark}

\begin{definition}\label{tilde zero}
 Modifying the $\leqslant$-preorder for $\S$-spans of {\em Definition \ref{retractive}} we define the relation $\leqslant^{\star}$ for $\S$-spans by
$$(s,f)\leqslant^{\star}(\tilde{s},\tilde{f})\Longleftrightarrow \exists\; x:(s,f)\longrightarrow(\tilde{s},\tilde{f}),\;
x\in\S^{\star}.$$
With the closure of $\S^{\star}$ under pullback (along morphisms in $\S$ when $\S$ satisfies the weak left cancellation condition) $\leqslant^{\star}$ is easily seen to be compatible relation for $\S$-spans\footnote{While $\leqslant^{\star}$ is trivially reflexive, we are, however, no longer being assured of its transitivity since $\S^{\star}$ may fail to be closed under composition.}, so that the least equivalence relation $\sim_{{\rm{z}}^{\star}}$
containing $\leqslant^{\star}$ is also compatible. Writing just ${\rm{z}}^{\star}$ when this modified zig-zag relation $\sim_{{\rm{z}}^{\star}}$ is used as an index, we call
the quotient category
$${\sf Par}(\C,\S):={\sf Span}_{{\rm{z}}^{\star}}(\C,\S)$$
the $\S$-{\em partial map category} of $\C$.
It comes with the functors
$$\Phi_{\rm{z}^{\star}}:\C\longrightarrow\; \sf{Par}(\C,\S)\;\longleftarrow\S^{\rm{op}}:\Psi_{\rm{z}^{\star}}$$
$$(f:D\to B)\longmapsto [1_D,f]_{\rm{z}^{\star}}\quad\quad [s,1_D]_{\rm{z}^{\star}} \longleftarrow\!\shortmid (s:A\leftarrow D).$$
 \end{definition}
 
 Let us point out right away that, as a consequence of Proposition \ref{properties}(3), when $\S$ is a class of monomorphisms, ${\sf Par}(\C,S)$ is (isomorphic to) the category ${\sf Span}(\C,\S)$, so that there is no clash with the standard terminology in this case (as alluded to in Corollary \ref{WPar in mono case}).
Next we wish to confirm that ${\sf Par}(\C,\S)$ is indeed a quotient of $\sf {Sect(\C,\S)}=\sf{Span}_{\rm{a}}(\C,\S)$ and prove:

 \begin{lemma}\label{comp rel vs zero connected}
The relation $\sim_{\rm{a}}$ of {\em Definition \ref{weak partial}} is contained in $\sim_{{\rm{z}}^{\star}}$.
\end{lemma}

 \begin{proof}
Employing again the notation used in the proof of Theorem \ref{weak partial universal}, when $(s,f)\approx_{\rm{a}}(t,g)$ we have an $\S$-cospan morphism $v:\langle\check{f},\check{s}\rangle\longrightarrow\langle\check{g},\check{t}\rangle$ with $v\in\S$, which gives us the (vertical) $\S$-span
morphisms $v^{\star}:(\hat{s},\hat{f})\longrightarrow(\hat{t},\hat{g})$ with $v^{\star}\in\S^{\star}$; consequently, $(\hat{s},\hat{f})\sim_{{\rm{z}}^\star}(\hat{t},\hat{g})$. Since $(s,f), (t,g)$ are obtained from $(\hat{s},\hat{f}),(\hat{t},\hat{g})$ by ``horizontal whiskering",
$(s,f)\sim_{{\rm{z}}^{\star}}(t,g)$ follows, by compatibility of $\sim_{{\rm{z}}^{\star}}$.
\end{proof}
The Lemma shows that the assignment
$[s,f]_{\rm{a}}\mapsto[s,f]_{\rm{z}^{\star}}$ describes a functor $$\Gamma:{\sf{Sect}}(\C,\S)\to{\sf{Par}}(\C,\S),$$ uniquely determined by $\Phi_{\rm{a}}\Gamma=\Phi_{\rm{z}^{\star}},\,\Psi_{\rm{a}}\Gamma=\Psi_{\rm{z}^{\star}}$ (see Theorem \ref{weak partial universal}). Consequently, ${\sf Par}(\C,\S)$ is a quotient category of ${\sf Sect}(\C,\S)$; but it is nothing new when $\S$ is a class of monomorphisms:

 \begin{corollary}\label{gamma}
{\em (1)} $\sf{Par}(\C,\S)\cong\sf {Sect}(\C,\S)/\!\sim,$\, with $\sim$ induced by $\Gamma$.

 {\em(2)} If $\S$ is a class of monomorphisms, then
$${\sf Par}(\C,\S)\cong{\sf Sect}(\C,\S)\cong{\sf Span}(\C,\S).$$
\end{corollary}

 Without any additional condition on $\S$ one can prove:

 \begin{theorem}\label{Par split rest}
	${\sf Par}(\C,\S)$ is a split restriction category, with its restriction structure defined by $$\overline{[s,f]}_{{\rm{z}^{\star}}}=[s,s]_{{\rm{z}}^{\star}},$$ for all $\S$-spans $(s,f)$.
\end{theorem}
\begin{proof}
Trivially, $(s,f)\leqslant^{\star} (t,g)$ implies $(s,s)\leqslant^{\star}(t,t)$. Thus, writing just $[s,f]$ for $[s,f]_{{\rm z}^{\star}}$ in what follows, $[s,f]=[t,g]$ implies $[s,s]=[t,t]$. As a consequence, $\overline{(\ \ )}$ is well-defined. We note that
	Lemmas \ref{section} and \ref{comp rel vs zero connected} imply $[s,1]\cdot[1,s]=1$, a crucial identity when we check (R1-4) below. Since trivially $[s,s]=[1,s]\cdot[s,1]$, the identity also shows that $[s,s]$, once recognized as a restriction idempotent, splits.
	
(R1) For every morphism $[s,f]$ one has	\begin{center}
\begin{tabular}{ll}
$[s,f]\cdot\overline{[s,f]}$&$=[s,f]\cdot[s,s]=[1,f]\cdot[s,1]\cdot[1,s]\cdot[s,1]$\\&$=[1,f]\cdot[s,1]=[s,f].$
\end{tabular}
\end{center}
	
(R2)	For morphisms $[s,f]$ and $[t,g]$ with the same domain, we form the pullback square
$s\cdot t'=t\cdot s'$ in $\C$ and obtain the needed equality below:
\begin{center}
\begin{tabular}{ll}
$\overline{[s,f]}\cdot\overline{[t,g]}$&$=[s,s]\cdot[t,t]=[t\cdot s',s\cdot t']$\\&$=[s\cdot t',t\cdot s']
=[t,t]\cdot[s,s]=\overline{[t,g]}\cdot \overline{[s,f]}.$\\
\end{tabular}
\hfil
\end{center}
	
(R3) With the same notation as in (R2), we have
$$\overline{[t,g]\cdot\overline{[s,f]}}=\overline{[t,g]\cdot[s,s]}=[s\cdot t',s\cdot t']=[t,t]\cdot[s,s]=\overline{[t,g]}\cdot \overline{[s,f]}.$$
	
(R4) For morphisms $[s,f]:A\rightarrow B$ and $[t,g]:B\rightarrow C$, we form the pullback square $t\cdot f'=f\cdot t'$ in $\C$ and obtain the needed equality below:
\begin{center}
\begin{tabular}{ll}
$\overline{[t,g]}\cdot[s,f]$&$=[t,t]\cdot[s,f]=[s\cdot t',t\cdot f']$\\
&$=[s\cdot t',f\cdot t']=[1,f]\cdot[s\cdot t',t']$\\
&$=[1,f]\cdot[s,1]\cdot[1,s]\cdot[s\cdot t',t']$\\
&$=[s,f]\cdot[s\cdot t',s\cdot t']=[s,f]\cdot\overline{[t,g]\cdot[s,f]}$.\\
\end{tabular}
\end{center}
\end{proof}

 Next we show that the functor $\Gamma$ of Corollary \ref{gamma} is a localizing functor, mapping the morphisms of the class $\Phi_{\rm{a}}(\S^{\star})$ (with $\Phi_{\rm{a}}: \C\longrightarrow{\sf Sect}(\C,\S),\, f\mapsto[1,f]_{{\rm a}},)$ to isomorphisms of ${\sf Par}(\C,\S)$, provided that $\mathcal S^{\star}\subseteq \mathcal S$: 

 \begin{theorem}\label{partial}
If $\mathcal S^{\star}\subseteq \mathcal S$, in particular if $\S$ satisfies the weak left cancellation condition, then ${\sf Par}(\C,\S)$ is a localization of ${\sf Sect(\C,\S})$:
$${\sf Par}(\C,\S)\cong{\sf Sect}(\C,\S)[\Phi_{\rm{a}}(\S^{\star})^{-1}].$$
Under the same hypothesis, ${\sf Retr}(\C,\S)$ is also a quotient category of ${\sf Sect}(\C,\S)$, and $${\sf Retr}(\C,\S)\cong\C[\S^{-1}]$$.
\end{theorem}

 \begin{proof}
For $x\in\S^{\star}$ we have $x\in \mathcal S$ by hypothesis and, therefore, $(x,x)\in \sf{Span}(\C,\mathcal S)$. Trivially then, $(x,x)\leqslant^{\star} (1,1)$, and $[x,x]_{\rm{z}^{\star}}=1$ follows. This implies $\Gamma[1,x]_{\rm a}\cdot \Gamma[x,1]_{\rm a}=\Gamma[x,x]_{\rm a}=1$, and since $[x,1]_{\rm a}\cdot[1,x]_{\rm a}=1$ by Lemma \ref{section}, we see that $\Gamma\Phi_{\rm{a}}x=\Gamma[1,x]_{\rm a}=[1,x]_{\rm{z}^{\star}}$ is an isomorphism, with inverse $\Gamma[x,1]_{\rm a}=[x,1]_{\rm{z}^{\star}}$.

 Now consider any functor $F:{\sf Sect}(\C,\S)\longrightarrow\mathcal D$ mapping all $\Phi_{\rm a}x \;(x\in\S^{\star})$ to isomorphisms. We must confirm that $F$ factors as $F'\Gamma = F$, for a unique functor $F':{\sf Par}(\C,\S)\longrightarrow\mathcal D$. But since $\Gamma$ is bijective on objects and full, this assertion becomes obvious once we have shown that $F'$ is well defined when (by necessity) putting $F'[s,f]_{{\rm z}^{\star}}:=F[s,f]_{\rm a}$ for all $\S$-spans $(s,f)$. Considering $(s,f)\leqslant^{\star}(t,g)$, so that $s=t\cdot x,\, f=g\cdot x$ for some $x\in\S^{\star}$, we first note that $[x,1]_{\rm a}\cdot[1,x]_{\rm a}=1$ implies
$F[1,x]_{\rm a}\cdot F[x,1]_{\rm a} = 1$ since $F[1,x]_{\rm a}$ is an isomorphism; consequently,
$$F[s,f]_{\rm a}=F[1,f]_{\rm a}\cdot F[s,1]_{\rm a}=F[1,g]_{\rm a}\cdot F[1,x]_{\rm a}\cdot F[x,1]_{\rm a}\cdot F[t,1]_{\rm a}=F[t,g]_{\rm a}.$$
Since $\leqslant^{\star}$ generates the equivalence relation $\sim_{{\rm z}^{\star}}$, well-definedness of $F'$ follows.

 The additional statement on the existence of quotient functors and on ${\sf Retr}(\C,\S)$ serving as a model for $\C[\S^{-1}]$ follows from the following obvious inclusions of the relevant equivalence relations: $\S^{\star}\subseteq\S$ implies $\sim_{{\rm z}^{\star}}\,\subseteq\, \sim_{\rm z}$ which, by Lemma \ref{comp rel vs zero connected}, gives $\sim_{\rm z}\,=\,\sim_{\rm az}$.
\end{proof}

There is an easy generalization of the main statement of Theorem \ref{partial} instead of $\S^{\star}$ one considers any pullback-stable subclass $\mathcal T$ of $\S$ which contains $\S^{\star}$. Rather than $\leqslant^{\star}$ we may then consider the $\S$-span relation
$$(s,f)\leqslant_{\mathcal T}(\tilde{s},\tilde{f})\Longleftrightarrow \exists\; x:(s,f)\longrightarrow(\tilde{s},\tilde{f}),\;
x\in{\mathcal T},$$
and its generated equivalence relation, the $\mathcal T$-{\em zig-zag relation} $\sim_{{\rm z}_{\mathcal T}}$. Hence, if we write just ${\rm z}_{\mathcal T}$ when $\sim_{{\rm z}_{\mathcal T}}$ is used as an index, an easy adaptation of the above proof then shows
$${\sf Span}_{{\rm z}_{\mathcal T}}(\C,\S)\cong{\sf Sect}(\C,\S)[\Phi_{\rm a}(\mathcal T)^{-1}].$$
Now, under the hypothesis $\S^{\star}\subseteq\S$, the choice ${\mathcal T}=\S^{\star}$ gives the first statement of Theorem \ref{partial}, while the choice $\mathcal T=\S$ returns Theorem \ref{fraction}, presenting $\C[\S^{-1}]$ as ${\sf Sect}(\C,\S)[\Phi_{\rm a}(\mathcal S)^{-1}].$

\section{{\sf Par} as a left adjoint 2-functor}
 Extending a key result obtained in \cite{colackrest} we now provide a setting which presents
$(\C,\S)\mapsto {\sf Par(\C,S)}$
as the left adjoint to the formation of the category ${\sf Total}(\X)$ for every split restriction category $\X$. In particular, the category ${\sf Par}(\C,\S)$ is thereby characterized by a universal property.

 Recall that, for a restriction category $\X$ with restriction operator $\overline{(-)}$, a morphism $f$ in $\X$ is called {\em total} if $\bar f=1$. As identity morphisms and composites of total morphisms are total, one obtains the category ${\sf Total}(\X)$, which has the same objects as $\X$. Any functor $F:\X\to\Y$ which preserves the restriction operations of the categories restricts to a functor $F:{\sf Total}(\X)\to{\sf Total}(\Y)$, and any (componentwise) total natural transformation $\alpha:F\to G$ of such functors keeps this role under the passage to total categories.

 Recall further that $i$ in $\X$ is a {\em restricted isomorphism} if, for some morphism $i^-$, one has $i^-\cdot i=\bar{i}$ and $i\cdot i^-=\overline{i^-}$; such $i^-$ is unique and called the {\em restricted inverse} of $i$ (also known as the {\em partial inverse} of $i$). We denote by ${\sf ReIso}(\X)$ 
 the class of restricted isomorphisms in $\X$ that are total. Remarkably, as shown in Proposition 3.3 of \cite{colackrest}, when $\X$ is a split restriction category, the pullback $j$ of $i\in {\sf ReIso(\X)}$ along any total morphism $f$ exists in ${\sf Total}(\X)$ and belongs to ${\sf ReIso}(\X)$ again: $j$ is part of the splitting of the restriction idempotent $\overline{\overline{i^-}\cdot f}=j\cdot r$ where $r\cdot j=1$, producing the pullback diagram
\begin{center}
$\xymatrix{\ar[r]^{i^-\cdot f\cdot j}\ar[d]_{j} & \ar[d]^{ i}\\
			\ar[r]_{ f} &}$
\end{center}

 As in \cite{colackrest}, but without any restriction to monomorphisms, we form the (very large) 2-category		
$$\mathfrak{StableCat}$$ of {\em stably structured categories}. Its objects are pairs $(\C,\mathcal S)$, where $\C$ is a category and $\S$ a class of morphisms satisfying the general hypotheses of Section 2; 
its morphisms $F:(\C,\mathcal S)\rightarrow (\D,\mathcal T)$ are functors $F:\C\rightarrow \D$
with $F(\mathcal S)\subseteq \mathcal T$ which
preserve pullbacks along $\mathcal S$-morphisms; 2-cells are natural transformations whose naturality squares involving $\S$-morphisms are pullback squares.

 $$\mathfrak{SplitRestCat}$$ denotes the (very large) 2-category of {\em split restriction categories}, with their restriction-preserving functors and total natural transformations. Then, as in \cite{colackrest}, we have the 2-functor
 \begin{center}
$\sf{Total}:\mathfrak{SplitRestCat}\longrightarrow \mathfrak{StableCat}\quad\quad\quad$
\end{center}

 \begin{center}
$\xymatrix{\X\ar@{|->}[rrr]\ar@<-1.5ex>[dd]_{F}^{\Rightarrow^\alpha}\ar@<1.8ex>[dd]^{G} &&&
({\sf Total}(\X),{\sf ReIso}(\X))\ar@<-2ex>[dd]^{\Rightarrow^{\alpha}}_{{\sf Total}(F)}\ar@<1.5ex>[dd]^{{\sf Total}(G)}\\
			&&&\\
			\Y\ar@{|->}[rrr]&&& ({\sf Total}(\Y),{\sf ReIso}(\Y))}$
\end{center}

 \noindent where ${\sf Total}(F)$ is the restriction of $F$, which we may write simply as $F$ again.

 Our aim is to show that there is a left adjoint that takes $(\C,\S)$ to ${\sf Par}(\C,\S)$. 
	It is straightforward to verify that every $F:(\C,\mathcal S)\rightarrow (\D,\mathcal T)$ in $\mathfrak{StableCat}$ gives us the well-defined restriction-preserving functor
\begin{center}
${\sf Par}(F):{\sf Par}(\C,\S)\longrightarrow {\sf Par}(\D,\T),\quad [s,f]_{{\rm z}^{\star}}\mapsto[Fs,Ff]_{{\rm z}^{\star}}.$
\end{center}
The resulting (ordinary) functor
${\sf Par}:\mathfrak{StableCat}\longrightarrow \mathfrak{SplitRestCat}$
is easily seen to be a 2-functor; it sends a 2-cell $\alpha:F\Rightarrow G$ to the total natural transformation $[1,\alpha]: {\sf Par}(F)\Rightarrow {\sf Par}(G)$ whose component at $A$ in $\C$ is defined by $[1,\alpha]_A=[1_{FA},\alpha_A]$. We claim that {\sf Par} is left adjoint to {\sf Total}:

 \begin{theorem}\label{2adjunction1} There is a 2-adjunction
	$${\sf Par}\dashv {\sf Total}:\mathfrak{SplitRestCat}\longrightarrow \mathfrak{StableCat}.$$
\end{theorem}

 \begin{proof}
	To construct the unit $\eta:{\rm Id}_\mathfrak{StableCat}\rightarrow \sf{Total}\circ\sf{Par}$
	at $(\C,\S)$ in $\mathfrak{StableCat}$, since in the notation of Section 4 
	the functor
	\begin{center}
		$\xymatrix{\C\ar[r]^{\Phi_{\rm a}\quad} & {\sf Sect}(\C,\S)\ar[r]^{\Gamma\quad} & {\sf Par}(\C,\S)},\;f\mapsto [1,f]=[1,f]_{{\rm z}^{\star}},$
	\end{center}
	has total values, we consider its restriction,
	$$\eta_{(\C,\S)}:\C \longrightarrow {\sf Total(Par}(\C,\S))$$
	and should show that $\eta_{(\C,\S)}$ lives indeed in
	$\mathfrak{StableCat}$. 
	For $s\in \mathcal S$, one clearly has $[1,s]\in{\sf ReIso(Par}(\C,\mathcal S))$. Furthermore, in the diagram 
		\begin{center}	
		$\xymatrix{\ar[r]^{f'}\ar[d]_{s'} & \ar[d]^{s}\\\ar[r]_{f} & }$\quad\quad\quad
		$\xymatrix{\ar[rr]^{[s,1][1,f][1,s']}\ar[d]_{[1,s']} && \ar[d]^{[1,s]}\\
			\ar[rr]_{[1,f]} && }$
	\end{center}
given the pullback square in $\C$ on the left one obtains the pullback square on the right,
living in the split restriction category ${\sf Par}(\C,\S)$. But, as one easily confirms, the top row of that pullback square equals $[1,f']$, so that the right diagram is in fact the $\eta_{(\C,\S)}$-image of the left diagram. Hence, $\eta_{(\C,\S)}$ preserves the relevant pullbacks.

For 1-cells $F, G:(\C,\S)\rightarrow (\mathcal D,\mathcal T)$ and a 2-cell $\alpha:F\Rightarrow G$, we need to show the commutativity of the following diagram, both at the 1-cell and 2-cell levels:
	\begin{center}
		$\xymatrix{\C\ar[rrr]^{\eta_{(\C,\S)}}\ar@<-2ex>[dd]_{F}^{\Rightarrow^\alpha}\ar@<2ex>[dd]^{G} &&& {\sf Total(Par}(\C,\S))\ar@<-3ex>[dd]^{\Rightarrow^{[1,\alpha]}}_{{\sf Total(Par}(F))}\ar@<3ex>[dd]^{{\sf Total(Par}(G))}\\
			&&&\\
			\mathcal D\ar[rrr]_{\eta_{(\mathcal D,\mathcal T)}}&&& {\sf Total(Par}(\mathcal D,\mathcal T))}$
	\end{center}
	
	\noindent Since for every morphism $f$ in $\C$ one has
		$$\eta_{(\mathcal D,\mathcal T)}(Ff)=[1,Ff]={\sf Par}(F)([1,f])
		={\sf Par}(F)(\eta_{(\C,\S)}(f)),$$
which shows the commuatativity at the 1-cell level:
		$$\eta_{(\mathcal D,\mathcal T)}\circ F={\sf Total(Par}(F))\circ\eta_{(\C,\S)}.$$
		At the 2-cell level, commutativity follows easily as well since, for all objects $A$ in $\C$,
		one has
		$$\eta_{(\mathcal D,\mathcal T)}( \alpha_A)=[1,\alpha_A]=[1,\alpha] _{\eta_{(\C,\S)}(A)}.$$

 	Next we define the counit $\varepsilon:\sf{Par}\circ\sf{Total}\rightarrow Id_{\mathfrak{SplitRestCat}}$.
	For a split restriction category $\X$, since $\sf{ReIso}(\X)$ is a collection of monomorphisms, one has $\sf{Par(Total}(\X),{\sf ReIso}(\X))\cong\sf{Span}(\sf{Total}(\X),\sf{ReIso}(\X))$ (see Corollary \ref{gamma}(2)), and one may define the functor $$\varepsilon_\X:\sf{Par(Total}(\X),\sf{ReIso}(\X))\rightarrow \X$$  as in  Theorem 3.4 of \cite{colackrest}, by	simply taking $[s,f]$ to $f\cdot s^-$. To confirm that $\varepsilon$ is 2-natural, we consider 1-cells $H,K:\X\to\Y$ of split restriction categories and a 2-cell $\beta:H\Rightarrow K$
and show the commutativity of the following diagram at both, the 1-cell and 2-cell levels:
	\begin{center}
		$\xymatrix{\sf{Par(Total}(\X),{\sf ReIso}(\X))\ar[rrr]^{\varepsilon_{\X}}\ar@<-3ex>[dd]_{{\sf Par(Total}(H))}^{\Rightarrow^{[1,\beta]}}\ar@<3ex>[dd]^{{\sf Par(Total}(K))} &&& \X\ar@<-2ex>[dd]_{H}^{\Rightarrow^{\beta}}\ar@<2ex>[dd]^{K}\\
			&&&\\
			\sf{Par(Total}(\Y),{\sf ReIso}(\Y))\ar[rrr]^{\varepsilon_{\Y}} &&& \Y}$
	\end{center}
	
	\noindent At the 1-cell level, for every morphism $[s,f]$ in ${\sf Par(Total}(\X),{\sf ReIso}(\X))$, we have
\begin{center}
\begin{tabular}{ll}	
	$\varepsilon_{\Y}({\sf Par}(H))([s,f]))$&$=\varepsilon_{\Y}([Hs,Hf])=Hf\cdot(Hs)^-$\\&$=Hf\cdot H(s^-)=H(f\cdot s^-)=H(\varepsilon_{\X}([s,f])).$
\end{tabular}
\end{center}
 At the 2-cell level, for every object $X$ in $\X$, we just note that
$$(\varepsilon_{\Y}[1,\beta])_X=\varepsilon_{\Y}([1,\beta_X])=\beta_X=\beta_{\varepsilon_{\X}(X)}=(\beta\varepsilon_{\X})_X.$$
	
	Finally, since the composite functor 
	$$\xymatrix{\sf{Total}(\X)\ar[rr]^{\hskip-15mm\eta_{\sf{Total}(\X)}}&& \sf{Total(Par(Total}(\X),{\sf ReIso}(\X)))\ar[rr]^{\hskip15mm\sf{Total}(\varepsilon_{\X})} && \sf{Total}(\X)}$$
	is described by $f\mapsto [1,f]\mapsto f$, the first triangular identity for the adjunction holds trivially.
For the second one, we see that the composite functor	
	$$\xymatrix{\sf{Par}(\C,\S)\ar[rr]^{\hskip-15mm\sf{Par}(\eta_{(\C,\S)})} &&
	\sf{Par(Total(Par}(\C,\S)),\I_{\C,\S})) \ar[rr]^{\hskip15mm\varepsilon_{\sf{Par}(\C,\S)}}	&& \sf{Par}(\C,\S)}$$
with $\I_{\C,\S}={\sf ReIso(Total(Par}(\C,\S)))$    
is described by	
	$$[s,f]\mapsto [[1,s],[1,f]]\mapsto [1,f][s,1]=[s,f],$$
	so that it maps identically as well.		
\end{proof}

We note that the counit $\varepsilon_\X$ at the split restriction category $\X$ as described in the above proof is actually an isomorphism (see  Theorem 3.4 of  \cite{colackrest}), so that $\mathfrak {SplitRestCat}$ {\em may be considered as a full reflective subcategory of} $\mathfrak{StableCat}$. Another important consequence of this fact is that, for every stably structured category $(\C,\S)$, the functor {\sf Par}$(\eta_{(\C,\S)})$ is an isomorphism. This proves:

\begin{corollary}
The partial map category {\sf Par}$(\C,\S)$ of a stably structured category $(\C,\S)$ may be isomorphically presented as the partial map category of $({\sf Total(Par}(\C,\S)),\I_{\C,\S})$ with $\I_{\C,\S}={\sf ReIso(Total(Par}(\C,\S)))$.
\end{corollary}

The Corollary tells us that, even without assuming $\S$-morphisms to be monic, we may interpret ${\sf Par}(\C,\S)$ as living in the Cockett-Lack \cite{colackrest} context. The paper \cite{colackrest} also tells us that, quite trivially, the unit $\eta_{(\C,\S)}$ is an isomorphism whenever $\S$ {\em is} a class of monomorphisms in $\C$. We can also show easily that the mono condition is actually necessary to make the unit an isomorphism.

\begin{corollary}\label{equivalence1}
The restriction of the 2-adjunction of {\em Theorem \ref{2adjunction1}} to its fixed objects is the Cockett-Lack 2-equivalence of ${\mathfrak{StrictRestCat}}$ with the full subcategory ${\mathfrak{StableCat}}_{\rm mono}$ of $\mathfrak{StableCat}$, given by the stably structured categories $(\C,\S)$ with $\S$ a class of monomorphisms in $\C$.
\end{corollary}

\begin{proof}
Assuming $\eta_{(\C,\S)}$ to be an isomorphism, we must show that every $s\in\S$ is a monomorphism in $\C$, and for that it suffices to see that $\eta_{\C,\S)}$ maps $s$ to a monomorphism in {\sf Total(Par}$(\C,\S))$. But $[s,1]\cdot[1,s]=1$ shows that $\eta_{(\C,\S)}(s)=[1,s]$ is a section and, hence, a monomorphism in {\sf Par}$(\C,\S)$, and it trivially maintains this status in the subcategory {\sf Total(Par}$(\C,\S))$.
\end{proof}

 \section{The split range category {\sf RaPar}$(\C,\S)$ of $\S$-partial maps in $\C$}
Range categories, as introduced by Cockett, Guo and Hofstra in \cite{cgh}, enhance the notion of restriction category, in the sense that, in addition to the restriction operator $\overline{(-)}$, they carry also a so-called {\em range operator} $\widehat{(-)}$, which behaves somewhat dually to the restriction operator, as follows:
\begin{definition}{\em (See \cite{cgh}.)}
	 A {\em range structure} on a restriction category is an assignment
	 \begin{center}
	 $f:A\longrightarrow B$
	
	 $\overline{\widehat f:B\longrightarrow B}$
	 \end{center}
of a morphism $\widehat{f}$ to each morphism $f$, satisfying the following four conditions:
\begin{itemize}
	\item[\em{(RR1)}] $\overline{\widehat f}=\widehat f$ for all morphisms $f$;
	\item[\em{(RR2)}] $\widehat f\cdot f=f$ for all morphisms f;
	\item[\em{(RR3)}] $\widehat{\bar g\cdot f}=\bar g\cdot\widehat f$ whenever ${\rm{codom}}(f) = \rm{dom}(g)$; 	
	\item[\em{(RR4)}] $\widehat{g\cdot\widehat f}=\widehat{g\cdot f}$ whenever $\rm{codom}(f) = \rm{dom}(g)$.
	\end{itemize}
	A restriction category equipped with a range structure is a {\em range category}; it is a {\em split range category} when it is split as a restriction category.
\end{definition}
Our goal is now to find a sufficiently large quotient of $\sf{Par}(\C,\S)$ which is a range category. To this end, throughout the rest of the paper, we assume that the class
\\\\{\indent \em $\S$ belongs to a {\em relatively stable} orthogonal factorization system} $(\mathcal P,\mathcal S),$\\\\
so that, in addition to having $\S$ being stable under pullback in $\C$, one has $\mathcal P$ being stable under pullback along $\S$-morphisms.
For every morphism $f$, we let $$f=s_f\cdot p_f$$ denote a (tacitly chosen) $(\mathcal P,\S)$-factorization.
As for every orthogonal factorization system, the general hypotheses on $\S$ as listed in Section 2, now come for free, and $\S$ is also weakly left cancellable (as defined in Proposition \ref{properties}(2)). Consequently, for the pullback-stable class $\S^{\star}$ of Section 6, one has $\S^{\star}\subseteq\S$. We denote by $$\S^{\circ}$$ the least pullback-stable class $\T$ with $\S^{\star}\subseteq\T\subseteq \S$ satisfying the additional $(\P,\S)$-{\em stability property}
$$\forall p, q\in\P,\,x\in\S,\,y\in\T\;(x\cdot q=p\cdot y\Longrightarrow x\in\T).$$
(Since this property, along with pullback stability, is stable under taking intersections and is trivially satisfied for $\T=\S$, there is such a class $\S^{\circ}$.)

 Proceeding as indicated at the end of Section 6, by setting $\T=\S^{\circ}$ there, 
 we can now define the desired quotient of $\sf{Par}(\C,\S)$ 
 and consider the zig-zag relation $\sim_{{\rm z}_{\S^{\circ}}}$, for which we write just ${\rm z}^{\circ}$ when used as an index. It is the least equivalence relation containing the relation $\leqslant_{\S^{\circ}}$, which we abbreviate as $\leqslant^{\circ}$.

 \begin{definition}\label{rpar}
We call
	$${\sf RaPar}(\C,\S):={\sf Span}_{{\rm z}^{\circ}}(\C,\S)$$
	the $\S$-{\em partial map range category} of $\C$.
\end{definition}

 Before confirming that this category is indeed a range category, we note that, since $\mathcal S^{\star}\subseteq\mathcal S^{\circ}$, we have the functor $$\Lambda:{\sf Par}(\C,\S)\rightarrow {\sf RaPar}(\C,\S),\quad[s,f]_{{\rm z}^\star}\mapsto[s,f]_{{\rm z}^{\circ}}.$$
Its induced equivalence relation presents its codomain as a quotient of its domain. Furthermore, with
$\Gamma$ as defined before Corollary \ref{gamma}, we obtain the first assertion of the following statement.
\begin{corollary}\label{localization func RaRe}
\begin{itemize}
	\item[{\em (1)}]$\sf{RaPar}(\C,\mathcal S)\!\cong\! \sf{Sect}(\C,\mathcal S)[\Phi_{\rm a}(\mathcal S^\circ)^{-1}]$ with localization,
$$\Lambda\Gamma\!:\!\sf{Sect}(\C,\mathcal S)\rightarrow \sf{RaPar}(\C,\mathcal S).$$
	\item[{\em (2)}] If $\S$ is a class of monomorphisms, then
	$${\sf RaPar}(\C,\S)\cong{\sf Par}(\C,\S)\cong{\sf Span}(\C,\S).$$
	\end{itemize}
\end{corollary}

 \begin{proof} (2) For $\S$ a class of monomorphisms, $\S^{\star}$ is the class of isomorphisms in $\C$ (by Proposition \ref{properties}(3)), which trivially satisfies the additional property defining $\S^{\circ}$ (since $\P$, dually to $\S$, satisfies the weak {\em right} cancellation property, and $\P\cap\S$ is the class of isomorphisms). Consequently, also $\S^{\circ}$ is the class of isomorphisms in $\C$.
\end{proof}

As a quotient of the split restriction category $\sf{Par}(\C,\S)$, $\sf{RaPar}(\C,\S)$ is a split restriction category too, with its restriction structure given by $$\overline{[s,f]}_{\rm z^{\circ}}=\Lambda[s,s]_{\rm z^{\star}}=[s,s]_{\rm z^{\circ}}$$
for all $\S$-spans $(s,f)$. Now we show:

 \begin{theorem}
	$\sf{RaPar}(\C,\mathcal S)$ is a split range category, with its range structure defined by
	$$\widehat{[s,f]}_{\rm z^{\circ}}=[s_f,s_f]_{\rm z^{\circ}}$$
	for all $\S$-spans $(s,f)$, where $s_f$ belongs to the $(\mathcal P,\S)$-factorization of $f=s_f\cdot p_f$.
\end{theorem}
\begin{proof}
To show that $\widehat{(\rm{-})}$ is well-defined, we consider $\S$-spans $(s,f), (t,g)$ with $(s,f)\leqslant^{\circ}(t,g)$, so that there exists a morphism $x\in\mathcal S^{\circ}$ with $s=t\cdot x,\,f=g\cdot x$. We have the diagonal morphism $d$ with $s_g\cdot d=s_f$ and $d\cdot p_f=p_g\cdot x$. By weak left cancellation, the first identity gives $d\in\S$, so that the second identity then implies $d\in\S^{\circ}$. Since $s_g\cdot d=s_f$, so that $(s_f,s_f)\leqslant^{\circ}(s_g,s_g)$, well-definedness of $\widehat{(\rm{-} )}$ follows.

 To now check (RR1-RR4), we write $[s,f]$ for $[s,f]_{\rm z^{\circ}}$.

 (RR1) holds trivially since $[s,s]$ is a restriction idempotent, for all $s\in\S$.

 (RR2) For an $\S$-span $(s,f)$ with $(\mathcal P,\S)$-factorization $f=s_f\cdot p_f$ and $(u,v)$ the kernel pair of $s_f$, one has
$$(s_f,s_f)\cdot(s,f)=(s_f,s_f)\cdot(1,s_f)\cdot(s,p_f)=(u,s_f\cdot v)\cdot(s,p_f)=(1,s_f)\cdot(u,v)\cdot(s,p_f)$$
 \noindent in $\sf{Span}(\C,\S)$. Since $[u,v]_{\rm a}=1$ by Lemma \ref{section}, also $[u,v]=[u,v]_{\rm z^{\circ}}=1$, and one concludes
$$\widehat{[s,f]}\cdot[s,f]=[s_f,s_f]\cdot[s,f]=[1,s_f]\cdot[u,v]\cdot[s,p_f]=[1,s_f]\cdot[s,p_f]=[s,f].$$

(RR3) For composable $\S$-spans $(s,f), (t,g)$ we must show $\widehat{\overline{[t,g]}\cdot[s,f]}=\overline{[t,g]}\cdot\widehat{[s,f]}$, where $\overline{[t,g]}=[t,t]$.
But the consecutive pullback diagrams
\begin{center}
$\xymatrix{\ar[r]^{p_f'}\ar[d]_{t''} & \ar[r]^{s_f'}\ar[d]_{t'} & \ar[d]^{t}\\
\ar[r]_{p_f}\ar@/_1.5pc/[rr]_{f} & \ar[r]_{s_f} &}$
\end{center}

\noindent in $\C$ and the $\S$-stability of $\mathcal P$ show $$\widehat{[t,t]\cdot[s,f]}=[t\cdot s_f', t\cdot s_f']=[s_f\cdot t',t\cdot s_f']=[t,t]\cdot[s_f,s_f]=[t,t]\cdot\widehat{[s,f]}.$$

 (RR4) Using the same notation as in (RR3) we just observe that the $\mathcal S$-part of the $(\mathcal P,\S)$-factorization of $g\cdot s_f'$ serves also as the $\mathcal S$-part of the $(\mathcal P,\S)$-factorization of $g\cdot s'_f\cdot p'_f$. But this observation implies immediately the desired equality $\widehat{[t,g]\widehat{[s,f]}}=\widehat{[t,g][s,f]}$.
\end{proof}

 The following chart summarizes our constructions under the provisions of this section:
{\small
$$\xymatrix{%
\S\subseteq\C\quad\quad \ar[d] && \cr
{\sf Span}(\C,\S) \ar[d] && \cr
{\sf Span_{\rm a}(\C,\S)} \ar[d]\ar@{=}[r] & {\sf Sect}(\C,\S) \ar[d]\ar@{-}[] & \cr
{\sf Span_{\rm z^{\star}}(\C,\S)} \ar[d]\ar@{=}[r] & {\sf Par}(\C,\S) \ar[d]\ar@{=}[r]^{\hs{-6}} & {\sf Sect}(\C,\S)[\Phi_{\rm a}(\S^{\star})^{-1}] \ar[d]  \cr
{\sf Span_{\rm z^{\circ}}(\C,\S)} \ar[d]\ar@{=}[r] & {\sf RaPar}(\C,\S) \ar[d]\ar@{=}[r] & {\sf Sect}(\C,\S)[\Phi_{\rm a}(\S^{\circ})^{-1}] \ar[d]  \cr
{\sf Span_{\rm z}(\C,\S)} \ar@{=}[r] & {\sf Retr}(\C,\S)\ar@{=}[r] & {\sf Sect}(\C,\S)[\Phi_{\rm a}(\S)^{-1}]= \C[\S^{-1}]={\sf Span}_{\rm az}(\C,\S) \cr
}$$
}
\section{{\sf RaPar} as a left adjoint 2-functor}

In analogy to Section 7, and in extension of one of the principal results obtained in \cite{cgh}, we now provide a setting which presents
$(\C,\S)\mapsto {\sf RaPar(\C,S)}$
as the left adjoint to the formation of the category ${\sf Total}(\X)$ for every split range category $\X$. This means in particular that the category ${\sf RaPar}(\C,\S)$ will be characterized by a universal property.


For a morphism $f$ in a split range category $\X$ with range operator $\widehat{(-)}$ one first notes that, if $\widehat{f}=1$, also the pullback of $f$ along a restricted isomorphism $i$ (see the first diagram of Section 7) satisfies $\widehat{i^-\cdot f\cdot j}=1$, with $j$ splitting the restriction idempotent $\overline{\overline{i^-}\cdot f}$.
Hence, as Theorem 4.7 of \cite{cgh} shows, the class $${\sf RaSur}(\X):=\{f\,|\,\bar{f}=1,\,\widehat{f}=1\,\}$$ of {\em range surjections} in ${\sf Total}(\X)$ is stable under pullback along ${\sf ReIso}(\X)$; moreover, $({\sf RaSur}(\X),{\sf ReIso}(\X))$ is an orthogonal factorization system of the category ${\sf Total}(\X)$.

 As in \cite{cgh}, but without any restriction to monomorphisms, we form the (very large) 2-category		
$$\mathfrak{StableFact}$$ of {\em relatively stable factorization systems}. Its objects are triples $(\C,\mathcal P,\mathcal S)$, where $\C$ is a category equipped with an orthogonal factorization system $(\mathcal P,\mathcal S)$, such that $\C$ has pullbacks along $\S$-morphisms and $\P$ is stable under them; its morphisms $F:(\C,\mathcal P,\mathcal S)\rightarrow (\D,\mathcal Q,\mathcal T)$ are functors $F:\C\rightarrow \D$
with $F(\mathcal P)\subseteq \mathcal Q$ and $F(\mathcal S)\subseteq \mathcal T$ which
preserve pullbacks along $\mathcal S$-morphisms; 2-cells are natural transformations whose naturality squares involving $\S$-morphisms are pullback squares.

 $$\mathfrak{SplitRangeCat}$$ denotes the (very large) 2-category of {\em split range categories}, with their range-preserving restriction functors and total natural transformations. Then, as in \cite{cgh}, we have the 2-functor
 \begin{center}
$\sf{Total}:\mathfrak{SplitRangeCat}\longrightarrow \mathfrak{StableFact}$
\end{center}
\begin{center}
$\quad\quad\quad\quad\quad\quad\X\mapsto ({\sf Total}(\X),{\sf RaSur}(\X),{\sf ReIso}(\X)),$
\end{center}

\noindent which, on 1- and 2-cells, is defined as in Section 7.


We want to show that there is a left adjoint, that takes $(\C,\P,\S)$ to ${\sf RaPar}(\C,\S)$. (We write ${\sf RaPar}(\C,\S)$ for ${\sf RaPar}(\C,\P,\S)$ 
since $\P$ is determined by $\C$ and $\S$.) For that, we first show the following essential Lemma, using an extension of the notation of Section 8:

 \begin{lemma}\label{circlemma}
For every functor $F\!:(\C,\mathcal P,\mathcal S)\rightarrow (\D,\mathcal Q,\mathcal T)$ in $\mathfrak{StableFact}$ one has $$F(\S^{\circ})\subseteq (F(\S^{\star}))^{\circ}\subseteq \T^{\circ},$$
where $(F(\S^{\star}))^{\circ}$
is the least pullback-stable class $\mathcal V$ in $\D$ with $F(\S^{\star})\subseteq{\mathcal V}\subseteq \T$ satisfying the $(\Q,\T)$-stability property.
\end{lemma}

 \begin{proof}
Since $F$ transforms pullbacks of $\S$-morphisms into pullbacks of $\T$-morphisms, for every morphism $v$ of $\S$-cospans one has (in the notation of Section 4) $F(v^{\star})=(Fv)^{\star}$. This implies $F(\S^{\star})\subseteq\T^{\star}$ and then $(F(\S^{\star}))^{\circ}\subseteq \T^{\circ}$.

 To prove the other inclusion claimed, for any class $\mathcal V$ as in the Lemma we form the class $\mathcal U = F^{-1}(\mathcal V)\cap\S$, which trivially satisfies $\S^{\star}\subseteq\mathcal U\subseteq\S$, as well as the $(\P,\S)$-stability property. Consequently, $\S^{\circ}\subseteq\mathcal U$, and then $F(\S^{\circ})\subseteq F(\mathcal U)\subseteq\mathcal V$. With this last inclusion holding for all $\mathcal V$,
$F(\S^{\circ})\subseteq (F(\S^{\star}))^{\circ}$ follows.
\end{proof}
	
As a consequence of Lemma \ref{circlemma}, every $F:(\C,\mathcal P,\mathcal S)\rightarrow (\D,\mathcal Q,\mathcal T)$ in $\mathfrak{StableFact}$ gives us the well-defined restriction- and range-preserving functor
\begin{center}
${\sf RaPar}(F):{\sf RaPar}(\C,\S)\longrightarrow {\sf RaPar}(\D,\T),\quad [s,f]_{{\rm z}^{\circ}}\mapsto[Fs,Ff]_{{\rm z}^{\circ}}.$
\end{center}
Defining it on 2-cells as {\sf Par} is defined in Section 7, we obtain the 2-functor
$${\sf RaPar}:\mathfrak{StableFactS}\longrightarrow \mathfrak{SplitRangeCat}$$
and can now claim:

 \begin{theorem}\label{2adjunction2} There is a 2-adjunction
	$${\sf RaPar}\dashv {\sf Total}:\mathfrak{SplitRangeCat}\longrightarrow \mathfrak{StableFact}.$$
\end{theorem}

 \begin{proof}
	The unit $\eta:{\rm Id}_\mathfrak{StableFact}\rightarrow \sf{Total}\circ\sf{RaPar}$
	at $(\C,\P,\S)$ in $\mathfrak{StableFact}$ is constructed as in Section 7. Indeed, since in the notation of Sections 4 and 8
	the functor
	\begin{center}
		$\xymatrix{\C\ar[r]^{\Phi_{\rm a}\quad} & {\sf Sect}(\C,\S)\ar[r]^{\Lambda\Gamma\quad} & {\sf RaPar}(\C,\S)},\;f\mapsto [1,f]=[1,f]_{{\rm z}^{\circ}},$
	\end{center}
	has total values, we may consider its restriction,
	$$\eta_{(\C,\P,\S)}:(\C,\P,\S)\longrightarrow {\sf Total(RaPar}(\C,\S)).$$
	To show that $\eta_{(\C,\P,\S)}$ lives in
	$\mathfrak{StableFact}$, beyond the proof of Theorem \ref{2adjunction1} we just have to note that, for $p\in \mathcal P$, $[1,p]$ is total and $\widehat{[1,p]}=1$, so that $[1,p]\in {\sf RaSur(RaPar}(\C,\mathcal S))$. 
Naturality of $\eta$ is established as in Theorem \ref{2adjunction1}.
	

 	To define the counit $\varepsilon:\sf{RaPar}\circ\sf{Total}\rightarrow Id_{\mathfrak{SplitRangeCat}}$, we may proceed as in Section 7 as well, Indeed,
	for a split range category $\X$, 
	we define the functor $$\varepsilon_\X:\sf{RaPar(Total}(\X),{\sf RaSur}(\X),{\sf ReIso}(\X))\rightarrow \X$$  as in  Theorem 3.4 of \cite{colackrest}, by	simply taking $[s,f]$ to $f\cdot s^-$. 
	
	All remaining verifications can proceed as in the proof of Theorem \ref{2adjunction1}.
\end{proof}

As a consequence of Theorem \ref{2adjunction2}, in analogy to the corresponding statements in Section 7, we can state 
that $\mathfrak {SplitRangeCat}$ {\em may be considered as a full reflective subcategory of} $\mathfrak{StableFact}$. Furthermore:

\begin{corollary}\label{equivalence2}
The restriction of the 2-adjunction of {\em Theorem \ref{2adjunction2}} to its fixed objects is the Cockett-Guo-Hofstra 2-equivalence of ${\mathfrak{StrictRangeCat}}$ with the full subcategory ${\mathfrak{StableFact}}_{\rm mono}$ of $\mathfrak{StableFact}$, given by categories $\C$ equipped with a relatively stable factorization system $(\P,\S)$, with $\S$ a class of monomorphisms in $\C$.
\end{corollary}

\section{Epilogue: The 2-category {\sf Span}$(\C,\S)$}
Following \cite{Benabou}, for $\C$ and $\S$ as in Section 2, we can set up the {\em bicategory} ${\sf Span}_{=}(\C,\S)$, with the same objects as those of $\C$; morphisms are spans (s,f) with $s\in\S$ (but with no isomorphic identification as in Section 2), and 2-cells $x: (s,f)\to(t,g)$ satisfy $x\cdot t =s,\,x\cdot g=f$; horizontal composition proceeds by (chosen) pullbacks, and vertical composition is as in $\C$. One has the pseudo-functors
$$\Phi_{=}=\Phi:\C\longrightarrow\; \sf{Span}_{=}(\C,\S)\;\longleftarrow\S^{\rm{op}}:\Psi=\Psi_{=}$$
$$(f:D\to B)\longmapsto (1_D,f)\quad\quad (s,1_D) \longleftarrow\!\shortmid (A\leftarrow D:s)$$
where $\Psi s$ is, up to isomorphism, determined by the adjunction $\Phi s\dashv\Psi s$, for all $s\in\S$. Under an obvious choice of pullbacks, the units and counits of these adjunctions are respectively given by the canonical 2-cells
$$\delta_s:(1_A,1_A)\to\Psi s\cdot\Phi s=(u,v)\quad{\rm{and}}\quad\varepsilon_s:\Phi s\cdot\Psi s=(s,s)\to (1_B,1_B),$$
for all $s:A\to B$ in $\S$, where $(u,v)$ denotes the kernel pair of $s$. As indicated in Theorem A.2 of \cite{Hermida} for the special case $\S=\C$, the pseudo-functor $\Phi$ is universal with the property of mapping $\S$-morphisms to {\em maps} (= spans that admit a right adjoint): any pseudo-functor $F:\C\to\D$ to a bicategory $D$ that sends $\S$-morphisms to maps and satisfies the (standard) Beck-Chevalley condition must factor as $H\Phi= F$, for a homomorphism $H:{\sf Span}_{=}(\C,\S)\to\D$ of bicategories that is unique up to isomorphism. Indeed, such homomorphism must preserve the adjunction $\Phi s\dashv\Psi s$, so that necessarily $$H(s,f)=H(\Phi f\cdot\Psi s)=H\Phi f\cdot H\Psi s=Ff\cdot Gs,$$ where $Gs:=H\Psi s$ is right adjoint to $Fs$.

From the bicategory ${\sf Span}_{=}(\C,\S)$ one obtains, as a quotient, the 2-category ${\sf Span}(\C,\S)={\sf Span}_{\cong}(\C,\S)$, whose horizontal ordinary category we have been considering throughout this paper, as follows: one declares the 2-cells $x:(s,f)\to(t,g),\,\tilde{x}:(\tilde{s},\tilde{f})\to(\tilde{t},\tilde{g})$ to be equivalent if there are isomorphisms $i,j$ in $\C$ with
$$\tilde{s}\cdot i= s,\, \tilde{f}\cdot i=f,\,\tilde{t}\cdot j=t,\,\tilde{g}\cdot j= g\quad{\rm{and}}\quad \tilde{x}\cdot i=j\cdot x;$$
the equivalence of the 2-cells $1_{(s,f)}$ and $1_{(\tilde{s},\tilde{f})}$ then means precisely that the $\S$-spans $(s,f)$ and $(\tilde{s},\tilde{f})$ are isomorphic, as defined in Section 2. The quotient categories ${\sf Retr}(\C,\S)$ and  ${\sf Sect}(\C,\S)$ may now respectively be seen as coming about by forcing the above counits $\varepsilon_s$ and units $\delta_s$ to become isomorphisms. Consequently, in 2-categorical terms, the characteristic property of the functors $\Phi_{\rm z}$ and $\Phi_{\rm a}$ of Sections 3 and 4 is that they turn $\S$-morphisms into {\em maps} that are retractions and, respectively, sections.

$$ $$

\noindent REFERENCES

\end{document}